\documentclass[final]{siamltex}
\usepackage{amsmath, amssymb, amsfonts}
\usepackage{graphicx,color}
\usepackage{diagbox}
\usepackage{algorithmic}
\usepackage{algorithm}
\usepackage{enumitem}
\usepackage{multirow}
\usepackage{subfig}
\usepackage{epstopdf}
\usepackage{comment}
\usepackage{braket}
\usepackage{adjustbox}

\newcommand{\Tr}{\text{Tr}}

\newcommand{\abs}[1]{\lvert#1\rvert}

\newcommand{\RR}{\mathbb{R}}
\newcommand{\CC}{\mathbb{C}}

\newtheorem{remark}[theorem]{Remark}

\global\long\def\ve{\varepsilon}
\global\long\def\R{\mathbb{R}}

\global\long\def\ra{\rightarrow}

\global\long\def\Tr{\text{Tr}}

\numberwithin{equation}{section}
\numberwithin{figure}{section}

\providecommand{\corollaryname}{Corollary}
\providecommand{\lemmaname}{Lemma}
\providecommand{\propositionname}{Proposition}
\providecommand{\remarkname}{Remark}
\providecommand{\theoremname}{Theorem}

\allowdisplaybreaks

\title{Semidefinite relaxation of multi-marginal \\ optimal transport for strictly correlated
electrons in second quantization}

\author{Yuehaw Khoo\thanks{Department of Statistics, University of Chicago, Illinois, IL 60637, USA. Email: {\tt ykhoo@uchicago.edu}
}
\and
Lin Lin\thanks{Department of Mathematics, University of California, Berkeley, and Computational Research Division, Lawrence Berkeley National Laboratory, Berkeley, CA 94720, USA.  Email: \texttt{linlin@math.berkeley.edu}}
\and
Michael Lindsey\thanks{Department of Mathematics, University of California, Berkeley, CA 94720, USA.  Email: \texttt{lindsey@math.berkeley.edu}}
\and
Lexing Ying\thanks{Department of Mathematics and Institute for Computational and Mathematical Engineering,
Stanford University, Stanford, CA 94305, USA. Email: {\tt lexing@stanford.edu}
}
}

\begin{document}

\maketitle

\begin{abstract}
We consider the strictly correlated electron (SCE) limit of the
fermionic quantum many-body problem in the second-quantized formalism.
This limit gives rise to a multi-marginal optimal transport (MMOT)
problem. Here the marginal state space for our MMOT problem is the
binary set $\{0,1\}$, and the number of marginals is the number $L$ of
sites in the model.  The costs of storing and computing the exact
solution of the MMOT problem both scale exponentially with respect to
$L$.  We propose an efficient convex relaxation to the MMOT which can be solved by
semidefinite programming (SDP). In particular, the semidefinite
constraint is only of size $2L\times 2L$.  We further prove that the SDP has dual attainment, in spite of the lack of Slater's condition (i.e.  the primal SDP does not have any strictly feasible point).  In the context of determining the lowest energy of electrons via density functional theory, such dual attainment implies the existence of an effective potential needed to solve a nonlinear Schr\"{o}dinger equation via self-consistent field iteration.  We demonstrate the effectiveness of our methods on computing the ground state energy of spinless and spinful Hubbard-type models. Numerical results indicate that our SDP formulation yields comparable results when using the unrelaxed MMOT formulation. We also describe how our relaxation methods generalize to arbitrary MMOT problems with pairwise cost functions.
\end{abstract}

\section{Introduction}

A central yet formidable task of quantum chemistry is to determine the ground-state energy of many electrons. Although many electronic structure theories have been proposed to tackle this problem with exponential complexity in number of electrons, the Kohn-Sham density functional theory (DFT)~\cite{HohenbergKohn1964,KohnSham1965} remains one of the most popular techniques due to its relatively cheap cost. The Kohn-Sham DFT (KS-DFT) reduces the computational complexity by approximating the Coulombic interaction term (more precisely, the exchange-correlation term) with a functional only depending on the density of the electrons. Therefore, the success of the widely-used KS-DFT hinges on the accuracy of the approximate functionals that capture the Coulombic interactions between the particles. 

In this paper, we consider solving the MMOT problem, which is related to a particular choice of the functional, the strictly correlated electron (SCE) functional $E_\text{sce}$ \cite{SeidlPerdewLevy1999}. More precisely, let $\rho\in \mathbb{R}^L$ be a density of $L$ discrete variables normalized to $N$:
\begin{equation}
\rho = [\rho_1,\ldots,\rho_L]^T\geq 0,\quad \sum_{p=1}^L \rho_p = N.
\end{equation}
The SCE functional is defined via the MMOT problem
\begin{eqnarray}
E_\text{sce}[\rho] &
:=& \min_{\mu\in \Pi(\rho)} \sum_{s_1,\ldots,s_L \in \{0,1\}} \sum_{p,q} C_{pq}(s_p,s_q)\mu(s_1,\ldots,s_L) 
\end{eqnarray}
where 
\begin{multline}
\Pi(\rho) = \{\mu\ \vert\  \mu\geq 0,\ \sum_{s_1,\ldots s_L\setminus s_p} \mu(s_1,\ldots,s_L) = \cr \rho_p(s_p)\delta_{s0}(s_p) + (1-\rho_p(s_p))\delta_{s1}(s_p),\ p=1,\ldots,L\},
\end{multline}
i.e. $\mu$ has marginal densities fixed according to $\rho$. For the case considered,
\begin{equation}
C_{pq}(s_p,s_q) = \begin{cases} v_{pq} & \text{if}\ s_p=s_q=1 \\ 0 & \text{otherwise}, \end{cases}
\end{equation}
for specific choice of $v_{pq}$'s. Note that the dimension of the feasible space for the MMOT problem is exponential in $L$, rendering infeasible any direct approach based on the formulation of $E_\text{sce}[\rho]$ as a general MMOT, at least for $L$ of moderate size. The purpose of this paper is thus to propose practical algorithms to evaluate $E_\text{sce}$ and $\nabla_\rho E_\text{sce}[\rho]$. Although the form of $C_{pq}$ and binary state space look rather restrictive, as we shall see the proposed methods can be easily extended to deal with more general cases.

Before moving on, we want to briefly discuss the significance of considering such a functional $E_\text{sce}$. In DFT, although tremendous progress has
been made in the construction of approximate functionals~\cite{PerdewZunger1981,Becke1988,LeeYangParr1988,PerdewBurkeErnzerhof1996},
these approximations are mostly derived by fitting known results for weakly correlated systems, for examples the uniform electron gas, single atoms, small molecules, and perfect crystal systems. Such functionals often perform well when the underlying quantum systems have single-particle energy that is significantly more important than the electron-electron interaction energy.  In order to extend the capability of DFT to the treatment of strongly correlated quantum systems, one recent direction of functional development
considers the limit in which the electron-electron interaction energy is infinitely large compared to other components of the total energy. The resulting limit is known as 
the SCE limit \cite{SeidlPerdewLevy1999,Seidl2007,ButtazzoDePascaleGori-Giorgi2012,MaletGori-Giorgi2012,CotarFrieseckeKlueppelberg2013,lewin2018semi,cotar2018smoothing}.
The SCE limit provides an alternative route to derive exchange-correlation energy functionals.
The study of Kohn-Sham DFT with SCE-based functionals is still in its infancy, but such approaches have already been used to treat strongly
correlated model systems and simple chemical systems (see e.g.~\cite{MendlMaletGori-Giorgi2014,ChenFrieseckeMendl2014,GrossiKooiGiesbertzEtAl2017}). 

\vspace{1mm}\noindent\textbf{Contribution:}

Based on the recent work of the authors~\cite{KhooYingOT}, we propose a convex relaxation approach by
imposing certain necessary constraints satisfied by the 2-marginals. The relaxed problem can be solved efficiently via
semidefinite programming (SDP).  While the 2-marginal formulation
provides a lower bound to the optimal cost of the MMOT problem, we also
propose a tighter lower bound obtained via an SDP involving the
3-marginals.  The computational cost for solving these relaxed problems is polynomial with respect to $L$, and, in particular, 
the semidefinite constraint is only enforced on a matrix of size $2L\times 2L$. 
Numerical results for spinless and spinful Hubbard-type
systems, demonstrate that the 2-marginal and 3-marginal relaxation
schemes are already quite tight, especially when compared to the
modeling error due to the Kohn-Sham SCE formulation itself.

By solving the dual problems for our SDPs, we can obtain the Kantorovich
dual potentials, which yield the SCE potential needed for
carrying out the self-consistent field iteration (SCF) in the Kohn-Sham
SCE formalism. To this end we need to show that the dual problem
satisfies strong duality and moreover that the dual optimizer is actually attained. We show that a straightforward
formulation of the primal SDP does not have any strictly feasible point,
and hence Slater's condition cannot be directly applied to establish
strong duality (see, e.g.,~\cite{BoydVandenberghe2004}). By a careful study of the structure of the dual problem, we prove that the strong duality and dual attainment conditions are indeed satisfied.   We also explain how the SDP relaxations introduced in this paper can be applied to arbitrary MMOT problems with pairwise cost functions. We comment that the justification of the strong duality and dual attainment conditions holds in this more general setting as well.

\vspace{1mm}\noindent\textbf{Related work:}

A system of $N$ interacting electrons in a $d$-dimensional space can be
described using either the first-quantized or the second-quantized
representation. In the first-quantized representation, the number of
electrons $N$ is fixed, and the electronic wavefunction is an
anti-symmetric function in $\bigwedge^{N} {L}^2(\RR^{d};\CC^2)$, which
is a subset of the tensor product space $\bigotimes^{N}
{L}^2(\RR^{d};\CC^2)$. Here $\CC^2$ corresponds to the spin degree of
freedom. In first quantization, the anti-symmetry condition needs to be
treated explicitly.  By contrast, in the second-quantized formalism, one
chooses a basis for a subspace of ${L}^2(\RR^{d};\CC^2)$. In practice,
the basis is of some finite size $L$, corresponding to a discretized
model with $L$ sites that encode both spatial and spin degrees of
freedom. The electronic wavefunction is an element of the Fock space
$\mathcal{F}\cong \CC^{2^{L}}$. The Fock space contains wavefunctions of
all possible electron numbers, and finding wavefunctions of the desired
electron number is achieved by constraining to a subspace of the Fock
space. In the second-quantized representation, the anti-symmetry
constraint is in some sense baked into the Hamiltonian operator instead
of the wavefunction, and this perspective often simplifies book-keeping
efforts. Due to the inherent computational difficulty of studying
strongly correlated systems such as high-temperature superconductors, it
is often necessary to introduce simplified Hamiltonians such as in
Hubbard-type models. These model problems are formulated directly in the
second-quantized formalism via specification of an appropriate
Hamiltonian.

To the extent of our knowledge, all existing works on SCE treat
electrons in the first-quantized representation with (essentially) a
real space basis. In this paper we aim at studying the SCE limit in the
second-quantized setting. Note that generally Kohn-Sham-type theories in
the second-quantized representation are known as `site occupation
functional theory' (SOFT) or `lattice density functional theory' in the physics
literature~\cite{SchoenhammerGunnarssonNoack1995,LimaOliveiraCapelle2002,CapelleCampo2013,SenjeanNakataniTsuchiizuEtAl2018,Coe2019}.
A crucial assumption of this paper is that the electron-electron interaction takes
the form $\sum_{p,q=1}^{L} v_{pq} \hat{n}_{p} \hat{n}_{q}$, which we call the generalized Coulomb interaction. (The meaning of the symbols will be explained in
Section~\ref{sec:prelim}.) We remark that the form of the generalized Coulomb interaction is more restrictive than the general
form $\sum_{p,q,r,s=1}^{L} v_{pqrs} \hat{a}^{\dagger}_{p} \hat{a}^{\dagger}_{q} \hat{a}_{s}\hat{a}_r$ appearing in the
quantum chemistry literature, to which our formulation does not yet apply.
Assuming a generalized Coulomb interaction, we demonstrate that the corresponding SCE problem can be formulated as a multi-marginal optimal transport (MMOT) problem over classical probability measures on the binary hypercube $\{0,1\}^L$. The cost function in this problem is of pairwise form. Hence the 
objective function in the Kantorovich formulation of the MMOT can be
written in terms of only the 2-marginals of the probability measure. In
order to solve the MMOT problem directly, even the storage cost of the
exact solution scales as $2^L$, and the computational cost also scales
exponentially with respect to $L$.  Thus a direct approach becomes impractical even when
the number of sites becomes moderately large.

In the first-quantized formulation, for a fixed real-space discretization the computational cost of the direct solution of the SCE problem   
scales exponentially with respect to the number of electrons $N$. This curse of dimensionality is a serious obstacle for 
SCE-based approaches to the quantum many-body problem. Although remarkably \cite{friesecke2018breaking,vogler2019kantorovich} show that the solution to the discretized SCE problem is sparse, finding such a sparse solution in a high dimensional space is still difficult. We note that there are exceptional cases, for example the strictly one-dimensional systems (i.e., $d=1$) and spherically symmetric systems (for any $d$)~\cite{Seidl2007}, for which semi-analytic solutions exist. 

In \cite{BenamouCarlierNenna2016}, the Sinkhorn scaling approach is applied to an entropically regularized MMOT problem. This
method requires the marginalization of a probability measure on a product space of size that is exponential in the number of electrons $N$. Thus the complexity of this method also scales exponentially with respect to $N$. Meanwhile, a method based on the Kantorovich dual of the MMOT problem was proposed in \cite{ButtazzoDePascaleGori-Giorgi2012,MendlLin2013}. However, there are exponentially many constraints in the dual problem. Furthermore,
\cite{ButtazzoDePascaleGori-Giorgi2012} assumes a Monge solution to the MMOT problem, but it is unknown whether the MMOT
problem with pairwise Coulomb cost has a Monge solution for $d=2,3$. Moreover, if it exists, the Monge solution is hard to evaluate in the context of the Coulomb cost. 

Another line of work that is similar in spirit to the proposed method focuses on reducing the search space of the MMOT problem to the set of $N$-representable 2-marginals in the setting of first quantization. Such a set is then relaxed to yield larger convex domains allowing for more efficient optimization. This type of method provides a lower bound to the SCE energy. In~\cite{friesecke2013n}, the $N$-representability constraint is relaxed to a $K$-representability constraint where $K< N$, and the resulting relaxation is solved by linear programming. To further improve the approximation quality, the aforementioned work~\cite{KhooYingOT} of two of the authors proposes a semidefinite relaxation-based approach to the MMOT problem arising from SCE in the first-quantized setting~\cite{KhooYingOT}. Furthermore, by proper treatment of the 3-marginal distributions, an upper bound to the SCE
energy is recovered as well. Numerical results indicate that both
the lower and upper bounds are rather tight approximations to
the SCE energy.

In order to use $E_\text{SCE}$ in the context of KS-DFT, the duality theory for the MMOT problem needs to be established. The line of work pursued in \cite{de2015optimal,di2017optimal,gerolin2019duality,colombo2019continuity} establishes the existence of dual maximizers for the repulsive MMOT problem in the first-quantized setting. As we shall see, in second quantization we only need to deal with a discrete MMOT problem where the existence of a dual maximizer is guaranteed by linear programming duality. However, the relaxation that we propose is no longer a linear program, so in this paper we must develop a duality theory for the proposed SDP.

In the second-quantized setting, our semidefinite relaxation-based approach for finding a lower bound to the
SCE energy is also related to the two-particle reduced density matrix
(2-RDM) theories in quantum chemistry~\cite{Coleman1963,Mazziotti1998,Mazziotti2004,Mazziotti2012,nakata2001variational}. However, the MMOT problem in SCE only requires the knowledge of the pair density instead of the entire 2-RDM. The number of constraints in our formulation is also considerably smaller than the number of constraints in 2-RDM theories, thanks to the generalized Coulomb form of the interaction.

\noindent\textbf{Organization:}
In Section~\ref{sec:prelim}, we briefly describe an appropriate formulation of Kohn-Sham DFT based on the SCE functional, which is in turn defined in terms of a MMOT problem. In Section~\ref{sec:two marginal
convex}, we solve the MMOT problem by
introducing a convex relaxation of the set of representable 2-marginals, and we prove strong duality for the relaxed problem. In
Section~\ref{sec:three marginal convex}, a tighter lower bound is
obtained by considering a convex relaxation of the set of representable 3-marginals.
In Section~\ref{sec:general OT}, we comment on how a general
MMOT problem with pairwise cost can be
solved by directly applying the methods introduced in  Sections~\ref{sec:two
marginal convex} and \ref{sec:three
marginal convex}. We demonstrate the effectiveness of the
proposed methods through numerical experiments in Section~\ref{sec:numeric}, and we discuss conclusions and future
directions in Section~\ref{sec:conclusion}.
For completeness, the background of KS-DFT and SCE functionals is introduced in the appendices.

\section{Density functional theory in second quantization}\label{sec:prelim}


In this section, we describe how the computation of $E_\text{sce}[\rho]$ and $\nabla_\rho E_\text{sce}[\rho]$ arises when computing the energy of second quantized electron using KS-DFT. Readers interested in the complete details are referred to the appendix. 

Let the wavefunctions of $N$ electrons be 
\begin{equation}
\Phi = \begin{bmatrix} \varphi_1& \ldots & \varphi_N \end{bmatrix} \in \mathbb{R}^{L\times N}.
\end{equation}
The KS-DFT states that the ground state energy of $N$ electrons can be obtained by solving a nonlinear Schr\"{o}dinger equation with a particular choice of effective potential $V_\text{eff}[\rho]$:
\begin{gather}
\label{eqn:KSSCE_eqn}
t \varphi_i+ \text{diag}\left(w+\nabla_\rho V_\text{eff}[\rho]\right)\varphi_i = \varepsilon_i \varphi_i,\quad i=1,\ldots,N\cr
\rho = \diag(\Phi \Phi^*),\quad \varphi^*_i \varphi_j = \delta_{ij},\ \forall i,j=1,\ldots,N.
\end{gather}
Here $t\in \mathbb{R}^{L\times L}$ is a kinetic energy operator that has the form of a graph adjacency matrix, where $t_{pq}\neq 0$ indicates possible hopping between site $p$ and $q$. The vector $w\in \mathbb{R}^L$ resembles an external potential on the electrons and the effective potential functional $V_\text{eff}[\rho]$ models the Coulombic interactions between the electrons. Although there are many choices for $V_\text{eff}[\rho]$, the goal of this paper is to consider the SCE limit where
\begin{equation}
    V_\text{eff}[\rho] = E_\text{sce}[\rho].
\end{equation}
In $E_\text{sce}[\rho]$, a nonzero $v_{pq}$ implies the existence of Coulombic repulsion between site $p$ and site $q$. Obtaining the derivatives $\nabla_\rho E_\text{sce}[\rho]$ amounts to solving for the dual optimizers of the MMOT via linear programming duality.

Eq.~\eqref{eqn:KSSCE_eqn} is a nonlinear eigenvalue problem and should be solved self-consistently. The standard iterative procedure for this task works
as follows:
\begin{enumerate}
\item For the $k$-th iterate $\rho^{(k)}$, compute $\nabla_\rho E_\text{sce}[\rho^{(k)}]$. 
\item Solve for $\{\varphi_i^{(k+1)}\}^L_{i=1}$ which satisfies:
\begin{gather}
t\varphi_i^{(k+1)} + \text{diag}\left(w+\nabla_\rho E_\text{sce}[\rho^{(k)}]\right)\varphi_i^{(k+1)} = \varepsilon_i \varphi_i^{(k+1)},\quad i=1,\ldots,N\cr
{\varphi_i^{(k+1)}}^* \varphi^{(k+1)}_j = \delta_{ij},\ \forall i,j=1,\ldots,N.
\end{gather}
and let $\rho^{(k+1)} := \mathrm{diag}(\Phi^{(k+1)} \Phi^{(k+1)*})$. 
\item Iterate until convergence, possibly using mixing
schemes~\cite{Anderson1965,Pulay1980,LinYang2013} to ensure or accelerate convergence. 
\end{enumerate}
Once self-consistency is reached and the iterations converge to $\rho^\star$, the total energy can be recovered by the
relation 
\begin{equation}
  E_{\text{KS-SCE}} = \sum_{k=1}^{N} \varepsilon_{k} - \nabla_\rho E_{\text{sce}}[\rho^\star]^T \rho^\star  + E_{\text{sce}}[\rho^\star].
  \label{eqn:KSSCEenergy}
\end{equation}
Readers interested in the background on how such a nonlinear eigenvalue problem arises are referred to Appendix~\ref{sec: background}. 

\section{Convex relaxation}
\label{sec:two marginal convex}
In this section, we discuss an SDP relaxation, $E^\text{sdp}_\text{sce}[\rho]$, to $E_\text{sce}[\rho]$ which has a polynomial problem size in $L$. Furthermore, we establish dual attainment for $E^\text{sdp}_\text{sce}[\rho]$ in order to show $\nabla_\rho E^\text{sdp}_\text{sce}[\rho]$ can indeed be obtained.

Let the 1-marginals
\begin{equation}
\mu^{(1)}_{p}(s_{p}) :=
\sum_{s_1,\ldots,s_L\backslash\{s_{p}\}} \mu(s_1,\ldots,s_L)
  \label{eqn:onemarginal}
\end{equation}
satisfy
\begin{equation}
  \mu^{(1)}_{p}(s) = (1-\rho_{p}) \delta_{s0} + \rho_{p} \delta_{s1}, \quad s=0,1
  \label{eqn:onemarginal_condition0}.
\end{equation}
We also have the 2-marginals $\mu^{(2)}_{pq}$ defined implicitly in terms of $\mu$ by marginalizing out all components other than $p,q$, i.e., by 
\begin{equation}
\mu^{(2)}_{pq}(s_{p},s_{q}) :=
\sum_{s_1,\ldots,s_L\backslash\{s_{p},s_{q}\}} \mu(s_1,\ldots,s_L).
  \label{eqn:twomarginal}
\end{equation}
which is identified as the $2\times 2$ matrix 
\begin{equation}
\mu^{(2)}_{pq} = \begin{bmatrix} \mu^{(2)}_{pq}(0,0) & \mu^{(2)}_{pq}(0,1) \\ \mu^{(2)}_{pq}(1,0) &
  \mu^{(2)}_{pq}(1,1) \end{bmatrix}.
\end{equation}

Despite the fact that it is possible to formulate $E_\text{sce}[\rho]$ as a general MMOT problem which is a linear program over $\mu$, the \emph{pairwise} cost: 
\[
C(s_1,\ldots,s_L) = \sum_{p\neq q} C_{pq} (s_p,s_q).
\]
allows one to write $E_\text{sce}[\rho]$ as
\begin{equation}
\label{2mar optimization}
E_{\text{sce}}[\rho] = \inf_{\mu \in \Pi(\rho)} 
\sum_{p\neq q} \sum_{s_p,s_q} C_{pq}(s_p,s_q)\mu^{(2)}_{pq}(s_p,s_q),
\end{equation}
or in matrix notation
\begin{equation}
\label{2mar optimization 2}
E_{\text{sce}}[\rho] = \inf_{\mu \in \Pi(\rho)} 
\sum_{p\neq q} \Tr[C_{pq} \mu^{(2)}_{pq}],
\end{equation}
where `Tr' indicates the matrix trace.

At first glance, it might seem that one may achieve a significant reduction of complexity
by directly changing the optimization variable in Eq.~\eqref{2mar optimization} from $\mu$ to
$\{\mu^{(2)}_{pq}\}_{p,q=1}^L$. However, extra constraints would then need to be
enforced in order to relate the different 2-marginals; i.e., the two-marginals must be jointly \emph{representable} in the sense that all of them could simultaneously be yielded from a single joint probability measure on $\{0,1\}^L$.

In this section, we show that a relaxation of the representability condition implicit in Eq.~\eqref{2mar optimization} allows us to formulate a tractable optimization problem in terms
of the $\{\mu^{(2)}_{pq}\}_{p,q=1}^L$ alone. In fact, this optimization problem will be a
semidefinite program (SDP).

\subsection{Primal problem}
We now derive certain necessary constraints satisfied by 2-marginals $\{\mu^{(2)}_{pq}\}_{p,q=1}^L$ that are obtained from a probability measure $\mu$ on $\{0,1\}^L$.
In the following we adopt the notation 
\[
\mathbf{s} = (s_1,\ldots,s_L) \in \{0,1\}^L.
\]
Then for any such $\mathbf{s}$, let $e_{\mathbf{s}}:\{0,1\}^L \ra \R$ be the Dirac probability mass function on $\{0,1\}^L$ localized at $\mathbf{s}$, i.e., 
\[
e_{\mathbf{s}}(\mathbf{s}') = \delta_{\mathbf{s},\mathbf{s}'}.
\]
Note that we can also write $e_{\mathbf{s}}$ as an $L$-tensor, i.e., an element of $\R^{2\times 2\times \cdots \times 2}$, via
\[
e_{\mathbf{s}} = e_{s_1} \otimes \cdots \otimes e_{s_L},
\]
where we adopt the (zero-indexing) convention $e_0 = [1,0]^\top$, $e_1 = [0,1]^\top$.

Any probability measure on $\{0,1\}^L$ can be written as a convex combination of the $e_{\mathbf{s}}$ since they are the extreme points of the set of probability measures; in particular we can write a probability density $\mu \in \Pi(\rho)$ as 
\begin{equation}
\mu = \sum_{\mathbf{s}} a_{\mathbf{s}} e_{\mathbf{s}}, \ \ \text{where}\ \  \sum_{\mathbf{s}} a_{\mathbf{s}} = 1,\  a_{\mathbf{s}} \geq 0.
\end{equation}
From the
definitions of the 1- and 2-marginals~\eqref{eqn:onemarginal}, \eqref{eqn:twomarginal}, it follows that
\begin{equation}
\mu^{(1)}_{p} = \sum_{\mathbf{s}}  a_{\mathbf{s}}\, e_{s_p}, \quad \mu^{(2)}_{pq} =
\sum_{\mathbf{s}}  a_{\mathbf{s}} \,e_{s_p}\otimes e_{s_q} = \sum_{\mathbf{s}}  a_{\mathbf{s}} \,e_{s_p} e_{s_q}^\top .
\label{eqn:tensorMarginals}
\end{equation}
Now define 
\begin{equation}
M = M(\{a_{\mathbf{s}}\}) = \sum_{\mathbf{s}} a_{\mathbf{s}} 
\begin{bmatrix} e_{\mathbf{s}_1} \\ \vdots \\ e_{\mathbf{s}_L} \end{bmatrix} \begin{bmatrix} e_{\mathbf{s}_1}^\top \cdots  e_{\mathbf{s}_L}^\top \end{bmatrix},
\end{equation}
Then by Eq.~\eqref{eqn:tensorMarginals}, $M$ is the matrix of $2\times 2$ blocks $M_{pq}$ given by 
\begin{equation}
M_{pq}=\begin{cases}
\mathrm{\mathrm{diag}}(\mu_{p}^{(1)}), & p=q,\\
\mu_{pq}^{(2)}, & p\neq q.
\end{cases}
\end{equation}
Accordingly we write $M = (M_{pq})  \in \mathbb{R}^{(2L)\times(2L)} $.
Then let $C = (C_{pq}) \in \R^{(2L)\times (2L)}$
be the matrix of the $2\times 2$ blocks $C_{pq}$ defined above, which
specifies the pairwise cost on each pair of marginals\footnote{Without loss of generality, one can assume $C_{pp} = 0$.}. Observe that the value of the objective function of Eq.~\eqref{2mar optimization 2} can in fact be rewritten as 
\[
\sum_{p\neq q} \Tr[C_{pq} \mu^{(2)}_{pq}] = \Tr[C M].
\] 
Then the MMOT problem Eq.~\eqref{2mar optimization 2} can be equivalently rephrased as
\begin{eqnarray}
E_{\text{sce}}[\rho]\  = \ & \underset{M \in \R^{(2L)\times(2L)}, \, \{a_{\mathbf{s}}\}_{\mathbf{s} \in \{0,1\}^L}}{\mathrm{minimize}} \ \ & \Tr(CM) \nonumber \\
&\text{subject to}\ \ & \ M = \sum_{\mathbf{s}} a_{\mathbf{s}} 
\begin{bmatrix} e_{\mathbf{s}_1} \\ \vdots \\ e_{\mathbf{s}_L} \end{bmatrix} \begin{bmatrix} e_{\mathbf{s}_1}^\top \cdots  e_{\mathbf{s}_L}^\top \end{bmatrix},  \label{eqn:Mconstraint} \\
&\ & \ 
M_{pp} = \mathrm{diag}(\mu_p^{(1)}) \,\text{ for all }\, p=1,\ldots,L, \cr
&\ & \ 
\sum_{\mathbf{s}} a_{\mathbf{s}} = 1,\quad a_{\mathbf{s}} \geq 0 \,\text{ for all }\, \mathbf{s} \in \{0,1\}^L. \nonumber
\end{eqnarray}
Note that in our application to SCE, we have fixed 
\[
\mu^{(1)}_p = \begin{bmatrix} 1- \rho_p \\ \rho_p \end{bmatrix}
\]
in advance, i.e., $\mu^{(1)}_p$ is \emph{not} an optimization variable.

At this point, our reformulation of the problem has not alleviated its exponential complexity; indeed, note that 
$\{a_{\mathbf{s}} \}_{\mathbf{s} \in \{0,1\}^L}$ is a vector of size $2^L$. However, the reformulation does suggest a way to reduce the complexity 
by accepting some approximation. In fact, we will 
omit $\{a_{\mathbf{s}} \}_{\mathbf{s} \in \{0,1\}^L}$ entirely from the optimization, retaining only $M$ as 
an optimization variable and enforcing several necessary constraints on $M$ that are
satisfied by the solution of the exact problem.

First, note from the constraint \eqref{eqn:Mconstraint} that $M$ is both
entry-wise nonnegative 
(written $M \geq 0$) and positive semidefinite (written $M\succeq 0$). Second, the fact that the 
1-marginals can be written in terms of the 2-marginals imposes 
additional \emph{local consistency} constraints on $M$. Indeed, with $\mathbf{1}_2 \in \mathbb{R}^2$ denoting the vector of all ones, we 
can write 
\begin{equation}
  \mu^{(2)}_{pq} \mathbf{1}_2 = \mu^{(1)}_{p}, \quad p\neq q,
  \label{}
\end{equation}
from which it follows that 
\begin{equation}
M_{pq}\mathbf{1}_2 = \begin{bmatrix} 1-\rho_p \\ \rho_p \end{bmatrix},\quad p,q=1,\ldots,L.
\end{equation}
Then we obtain the relaxation
\begin{eqnarray}
E_{\text{sce}}[\rho] \ \geq \ E_{\text{sce}}^{\text{sdp}}[\rho] \ :=\  & \underset{M \in \R^{(2L)\times(2L)}}{\mathrm{minimize}} \ \ & \Tr(CM) \label{eqn:sdp2marPre} \\
&\text{subject to}\ \ & \ M \succeq 0, \cr
&\ & \ M_{pq} \geq 0 \,\text{ for all }\, p,q=1,\ldots,L\ (p\neq q), \cr
&\ & \ M_{pq} \mathbf{1}_2 = \mu_p^{(1)} \,\text{ for all }\, p,q=1,\ldots,L\ (p\neq q), \cr
&\ & \ 
M_{pp} = \mathrm{diag}(\mu_p^{(1)}) \,\text{ for all }\, p=1,\ldots,L. \nonumber
\end{eqnarray}
Again, $\mu^{(1)}_p$ is \emph{not} an optimization variable.
It is actually helpful to reformulate the primal 2-marginal SDP \eqref{eqn:sdp2marPre} as
\begin{eqnarray}
\ E_{\text{sce}}^{\text{sdp}}[\rho] \ =\  & \underset{M \in \R^{(2L)\times(2L)}}{\mathrm{minimize}} \ \ & \Tr(CM) \label{eqn:sdp2mar} \\
&\text{subject to}\ \ & \ M \succeq 0, \label{eqn:psdCon} \\
&\ & \ M_{pq} \geq 0 \,\text{ for all }\, p,q=1,\ldots,L\ (p < q), \label{eqn:nnCon} \\
&\ & \ M_{pq} \mathbf{1}_2 = \mu_p^{(1)} \,\text{ for all }\, p,q=1,\ldots,L\ (p < q), \label{eqn:mar1Con} \\
&\ & \ M_{pq}^\top \mathbf{1}_2 = \mu_q^{(1)} \,\text{ for all }\, p,q=1,\ldots,L\ (p < q), \label{eqn:mar2Con} \\
&\ & \ 
M_{pp} = \mathrm{diag}(\mu_p^{(1)}) \,\text{ for all }\, p=1,\ldots,L. \label{eqn:diagCon}
\end{eqnarray}
Note that this formulation is equivalent to  \eqref{eqn:sdp2marPre}, given the 
symmetry of $M$ (implicit in the notation $M\succeq 0$). However, the new formulation removes 
a few redundant constraints and will help us 
derive a more intuitive dual problem. The problem \eqref{eqn:sdp2mar} will be referred to as the primal 2-marginal
SDP, or the \textit{primal problem} for short. Note that the optimal value of the primal problem is 
in fact attained because the constraints \eqref{eqn:psdCon}-\eqref{eqn:diagCon} define a compact feasible set.

Reflecting back on the derivation, we caution that replacing $E_\text{sce}[\rho]$ with 
$E_{\text{sce}}^{\text{sdp}}[\rho]$ comes
at a price. Since we only enforce certain necessary conditions on $M$, the 2-marginals 
that we recover from $M$ may not in fact be the 2-marginals of a joint probability measure on $\{0,1\}^L$. 
Thus $E_{\text{sce}}^{\text{sdp}}[\rho]$ should in general only be expected to be a 
lower-bound to $E_\text{sce}[\rho]$, though we will see that the error is often small in practice.

\subsection{Dual problem}
\label{sec:dual}
As detailed in Section \ref{sec:prelim}, in order to implement the SCF for Kohn-Sham SCE  
it is necessary to compute $\nabla_\rho E_\text{sce}[\rho]$. After replacing the density functional 
$E_\text{sce}[\rho]$ with the efficient approximation $E_\text{sce}^{\text{sdp}}[\rho]$, the 
same derivation motivates us to compute $\nabla_\rho E_\text{sce}^{\text{sdp}}[\rho]$. 
This quantity can obtained by examining the convex duality of our primal 2-marginal SDP.

We let $Y \succeq 0$ be the variable dual to the constraint \eqref{eqn:psdCon}, 
$Z_{pq} \geq 0 $ be dual to \eqref{eqn:nnCon},
$\phi_{pq}$ be dual to \eqref{eqn:mar1Con},
$\psi_{pq}$ be dual to \eqref{eqn:mar2Con}, 
and finally let 
$X_p $ be dual to \eqref{eqn:diagCon}.
Note that $Z_{pq} \in \mathbb{R}^{2\times 2}$ and $\phi_{pq}, \psi_{pq} \in \R^2$
for each $p<q$, and $X_p \in \R^{2\times 2}$ for each $p$.

Then our formal Lagrangian is of the form 
\[
\mathcal{L}\left( M,Y,\{Z_{pq}, \phi_{pq}, \psi_{pq} \}_{p<q}, \{X_p\} \right),
\]
where the domains of $M$ is the set of symmetric $2L\times 2L$ matrices (equivalently, it is 
convenient to think of $M$ as depending only on its upper-block-triangular part), and the 
dual variables are as specified above (i.e., only $Y \succeq 0 $ and $Z_{pq} \geq 0 $ 
are constrained), and more specifically we have (omitting the arguments of $\mathcal{L}$ from 
the notation)
\begin{eqnarray}
\mathcal{L} &= & \Tr(CM) - \Tr(YM) \label{eqn:lagrangian} \\
& & \ \ -\  2 \sum_{p<q} \left[ \Tr(Z_{pq}^\top M_{pq}) + \phi_{pq}^\top \left( M_{pq}\mathbf{1}_2 -  \mu^{(1)}_p  \right) 
+\psi_{pq}^\top  \left( M_{pq}^\top \mathbf{1}_2 -  \mu^{(1)}_q  \right) 
\right] \cr
& & 
\ \ -\  \sum_p \Tr \left(X_{p}^\top \left[M_{pp} - \mathrm{diag}(\mu_p^{(1)}) \right] \right). \nonumber
\end{eqnarray}
It is helpful to realize the identities
\[
\phi_{pq}^\top M_{pq}\mathbf{1}_2 = \Tr \left(M_{pq} [\mathbf{1}_2 \phi_{pq}^\top] \right),
\quad
\psi_{pq}^\top  M_{pq}^\top \mathbf{1}_2 = \Tr \left(M_{pq} [ \psi_{pq} \mathbf{1}_2^\top ]\right).
\]
Then, recognizing that $C=C^\top$ and $Y=Y^\top$ (so that $C_{pq}^\top = C_{qp}$ and $Y_{pq}^\top = Y_{qp}$), 
 minimization over $M$ of the Lagrangian \eqref{eqn:lagrangian} yields the dual problem
\begin{eqnarray}
& \underset{ Y, \, \{Z_{pq}, \phi_{pq}, \psi_{pq} \}_{p< q}, \,\{X_p\}}{\mathrm{maximize}} \ \ & 
\sum_p \Tr\left( X_p^\top \, \mathrm{diag}(\mu_p^{(1)}) \right) + 
2 \sum_{p<q} \left( \phi_{pq}^\top \mu_p^{(1)} + \psi_{pq}^\top \mu_q^{(1)} \right) 
\nonumber \\
&\text{subject to}\ \ & \ Y \succeq 0,  \nonumber \\
&\ & \ Z_{pq} \geq 0 \,\text{ for }\, p<q, \label{eqn:ZpqCon}\\
&\ & \ C_{pq} - Y_{pq} - Z_{pq} -  \phi_{pq}\mathbf{1}_2^\top - \mathbf{1}_2 \psi_{pq}^\top= 0
 \,\text{ for }\, p<q, \label{eqn:KantConPre}\\
&\ & \ C_{pp} - Y_{pp} - X_p^\top = 0. \label{eqn:XCon}
\end{eqnarray}
Observe that the variables $Z_{pq}$ can be removed by combining constraints 
\eqref{eqn:ZpqCon} and \eqref{eqn:KantConPre} to yield
\[
C_{pq} - Y_{pq} -  \phi_{pq}\mathbf{1}_2^\top - \mathbf{1}_2 \psi_{pq}^\top \geq 0.
\]
Moreover, $X_p$ can be removed simply by substituting $X_p = - Y_{pp}$ 
into the objective function (recall that $C_{pp} = 0$). These reductions yield
\begin{eqnarray}
& \underset{ Y, \, \{\phi_{pq}, \psi_{pq} \}_{p< q}}{\mathrm{maximize}} \ \ & 
2 \sum_{p<q} \left( \phi_{pq} \cdot \mu_p^{(1)} + \psi_{pq} \cdot \mu_q^{(1)} \right) 
- \sum_{p,s} Y_{pp}(s,s)\mu_p^{(1)}(s)
\label{eqn:sdp2marDual} \\
&\text{subject to}\ \ & \ Y \succeq 0, \label{eqn:sdp2marDualY}  \\
&\ & \ \phi_{pq}\mathbf{1}_2^\top + \mathbf{1}_2 \psi_{pq}^\top \leq C_{pq} - Y_{pq}
 \,\text{ for }\, p<q. \label{eqn:sdp2marDualPotentials} 
\end{eqnarray}
Here we think of $Y_{pp}(s,s)$ as the $(s,s)$ entry of the $2\times 2$ matrix $Y_{pp}$, and likewise 
$\mu_p^{(1)}(s)$ is the $s$-th entry of $\mu_p^{(1)}$. 

The dual problem may be interpreted as follows. Observe that for $Y$ fixed (e.g., fixed to its optimal value), the maximization problem decouples into a set of independent 
maximization problems for each pair of marginals. We think of $\widetilde{C}_{pq} := C_{pq} - Y_{pq}$ 
as defining an \emph{effective} cost function for each pair of marginals. Then the decoupled problem for a pair
$p<q$ is \emph{exactly} the Kantorovich dual problem in standard (i.e., not multi-marginal) optimal transport, specified 
by cost function $\widetilde{C}_{pq}$ and marginals $\mu_p^{(1)}$, $\mu_q^{(1)}$ \cite{villani2008optimal}. 
In other words, after fixing $Y$, our problem decouples into independent \emph{standard} optimal transport problems 
for each pair of marginals. Nonetheless, these problems are in turn themselves coupled via the optimization 
over $Y \succeq 0 $.

Recall that we wanted to compute $\nabla_\rho E_{\text{sce}}^{\text{sdp}} [\rho]$. Assuming that strong duality 
holds, as shall be established later,  the optimal value of the dual problem \eqref{eqn:sdp2marDual} 
is in fact equal to $E_{\text{sce}}^{\text{sdp}} [\rho]$. (Recall that here we think of the 
1-marginals $\mu_p^{(1)} = [1-\rho_p, \rho_p]^\top$ as being defined in terms of $\rho$.) 
Hence we can compute derivatives by evaluating the 
gradient of the objective function \eqref{eqn:sdp2marDual} with respect to $\rho$ at the \emph{optimizer} 
$\left(Y, \, \{\phi_{pq}, \psi_{pq} \}_{p\neq q}\right)$. (If the optimizer is not unique, then in general we will 
get a subgradient \cite{rock}.)

To carry out this program, first note that $\frac{\partial}{\partial \rho_r} \mu_p^{(1)} = \delta_{pr} [-1,1]^\top$. Therefore 
the partial derivative of the objective function \eqref{eqn:sdp2marDual} with respect to $\rho_r$ yields
\[
\frac{\partial E_{\text{sce}}^{\text{sdp}} [\rho]}{\partial \rho_r} = 
2\sum_{q>r} [\phi_{rq}(1)-\phi_{rq}(0)] + 2\sum_{p<r} [\psi_{pr}(1)-\psi_{pr}(0)] - [Y_{rr}(1,1) - Y_{rr}(0,0)].
\]
If one extends the definition of $\phi_{pq}, \psi_{pq}$ to $p>q$ via the stipulation $\phi_{pq} = \psi_{qp}$, then one has
\[
\frac{\partial E_{\text{sce}}^{\text{sdp}} [\rho]}{\partial \rho_r} = 
\sum_{p\neq r} [\phi_{rp}(1)-\phi_{rp}(0)] - [Y_{rr}(1,1) - Y_{rr}(0,0)].
\]

\subsection{Strong duality and dual attainment}
\label{sec:strong}

In this subsection, we show that the optimizer of the dual problem ~\eqref{eqn:sdp2marDual} can be attained. Before embarking on a proof of dual attainment, we note that Sion's minimax theorem \cite{Komiya1988} guarantees that the duality gap is zero as the domain of the primal problem is compact, as stated in the following Lemma.
\begin{lemma}
\label{lem:strongDuality}
The primal and dual problems~\eqref{eqn:sdp2mar} and \eqref{eqn:sdp2marDual}, respectively, have the same (finite) optimal value.
\end{lemma}

However, in order to compute the SCE potential, we actually require not only that the duality gap is zero, but also that the supremum in the dual problem is \emph{attained}. One might hope to verify Slater's condition~\cite{BoydVandenberghe2004}, which provides a standard method for verifying both strong duality and such `dual attainment' simultaneously. The trouble is that Slater's condition requires the existence of a feasible \emph{interior} point $M$, i.e., a point $M$ satisfying $M\succ 0$ and $M_{pq}>0$ for all $p\neq q$. This scenario is in fact impossible
since for example the vector
\begin{equation}
\begin{bmatrix}\mathbf{1}_2^\top & -\mathbf{1}_2^\top& 0 & \cdots & 0
\end{bmatrix}^\top \in \mathbb{R}^{2L}
\end{equation}
lies in the null space of any feasible $M$, hence $M\succ 0$ \emph{never} holds for feasible $M$.

Instead of using Slater's condition, we will prove dual attainment via a
very careful study of the structure of the dual problem.

\begin{theorem}
\label{thm:strongDuality}
The optimal value of the dual 2-marginal SDP \eqref{eqn:sdp2marDual} is attained. By Lemma \ref{lem:strongDuality}, this optimal value is equal to the optimal value of the primal 2-marginal SDP \eqref{eqn:sdp2mar}.
\end{theorem}
\begin{proof}
Without loss of generality we assume 
  \begin{equation}
    0<\rho_p<1, \quad p=1,\ldots,L.
    \label{eqn:rho_assumption_0}
  \end{equation}
To see why this assumption can be made, observe that if $\rho_p \in \{0,1\}$ for some $p$, then attainment for the dual problem \eqref{eqn:sdp2marDual} can be reduced to attainment for a strictly smaller dual 2-marginal SDP. We leave further details of such a reduction to the reader.) Also, 
for later reference, we let $F(Y, \{\phi_{pq},\psi_{pq}\}_{p<q})$ denote the objective function \eqref{eqn:sdp2marDual}, 
and we let $\mathcal{D}$ denote the feasible domain defined by the constraints \eqref{eqn:sdp2marDualY},  \eqref{eqn:sdp2marDualPotentials}.

Now to get started, observe that if we fix $Y \succeq 0$ and view \eqref{eqn:sdp2marDual} 
as an optimization problem over $\{\phi_{pq},\psi_{pq}\}_{p<q}$ only, the resulting problem 
is in fact a linear program. Let us call this the $Y$-program, more specifically: 
\begin{eqnarray*}
& \underset{ \{\phi_{pq}, \psi_{pq} \}_{p< q}}{\mathrm{maximize}} \ \ & 
2 \sum_{p<q} \left( \phi_{pq} \cdot \mu_p^{(1)} + \psi_{pq} \cdot \mu_q^{(1)} \right) 
- \sum_{p,s} Y_{pp}(s,s)\mu_p^{(1)}(s)
\label{eqn:Yprog} \\
&\text{subject to}\ \ & \ \phi_{pq}\mathbf{1}_2^\top + \mathbf{1}_2 \psi_{pq}^\top \leq C_{pq} - Y_{pq}
 \,\text{ for }\, p<q. \label{eqn:YprogCon} 
\end{eqnarray*}
In fact we may consider the $Y$-program for \emph{any} matrix $Y$, and this will slightly simplify some discussion later. 
Observe that each $Y$-program is feasible, and the optimal values $f(Y)$ of all $Y$-programs
are finite. 
Since they are linear programs, this means that the optimal values of
the $Y$-programs can be attained. 
Thus for each $Y$, 
there exist $\phi_{pq}^\star (Y)$, $\psi_{pq}^\star (Y)$ for $p <q$ which optimize the $Y$-program, i.e., attain 
the value $f(Y)$. 
By construction $f(Y)$ is concave, hence continuous, in $Y$.

Now let $d_0 = f(0)$, so  $d^\star \geq d_0$, where $d^\star$ is the optimal value of the dual problem \eqref{eqn:sdp2marDual}. 
Hence the feasible set of  \eqref{eqn:sdp2marDual} could be refined to $S\cap \mathcal{D}$, where 
\[
S:= \{  Y \succeq 0 \,: \, f(Y) \geq d_0 \},
\]
without altering the optimal value. Now if $S$ were compact, then the lemma would follow. To see this, 
note that since $d^\star < \infty$ (which follows from weak duality), we could take an optimizing sequence $(Y^{(k)}, \{\phi_{pq}^{(k)},\psi_{pq}^{(k)}\}_{p<q})$ for 
\eqref{eqn:sdp2marDual}, where $Y^{(k)} \in S \cap \mathcal{D}$. Then by compactness we could find a subsequence of $Y^{(k)}$ converging to some 
$Y^\star$. By the continuity of $f$, then $f(Y^\star) = d^\star$. Then it would follow that the optimum is attained at 
the point 
$(Y^\star, \{\phi_{pq}^\star (Y^\star),\psi_{pq}^\star (Y^\star)\}_{p<q})$.

Unfortunately, $S$ is not compact, but we will find a further constraint that does yield a compact feasible set without 
altering the optimal value. Then the preceding argument will complete the proof.

To further constrain the feasible set, we will observe a transformation of $Y$ that preserves the value of $f(Y)$, 
then `mod out' by this transformation. 
To this end, first note that via the discussion of Kantorovich duality following \eqref{eqn:sdp2marDual}
we can in fact write 
\[
f(Y)=-\sum_{p=1}^{L}\Tr\left[Y_{pp}\mathrm{diag}(\mu_{p}^{(1)})\right]+\sum_{p,q=1}^{L}\mathbf{OT}_{pq}(C_{pq}-Y_{pq}),
\]
 where $\mathbf{OT}_{pq}(A)$ is the optimal cost of the \emph{standard}
 optimal transport problem with cost matrix $A$ and marginals $\mu_{p}^{(1)},\mu_{q}^{(1)}$.

Then let $P\in \mathbb{R}^{(2L)\times (L-1)}$ be defined by 
\begin{equation}
P := \begin{bmatrix} \mathbf{1}_2 & & & \\ -\mathbf{1}_2 &  \mathbf{1}_2 & &  \\  & -\mathbf{1}_2& \ddots &  \\  &  & \ddots & \mathbf{1}_2  \\  & & &  -\mathbf{1}_2 \end{bmatrix},
\end{equation} 
and let its columns be denoted $P_i$ for $i=1,\ldots, L-1$. Then we claim that 
\begin{equation}
f(Y) = f\left( Y + P_i v^\top + v P_i^\top \right) 
\label{eqn:fClaim}
\end{equation}
for any $Y$, $v\in \R^{2L}$, and any $i=1,\ldots,L-1$. To prove this, write 
\[
v = \left[v_1^\top \cdots v_L^\top \right]^\top,
\]
where $v_q \in \R^2$ for $q=1,\ldots,L$.
Then observe that, via the discussion of Kantorovich duality following the statement \eqref{eqn:sdp2marDual} of
the dual problem,
we can in fact write 
\[
f(Y)=-\sum_{p=1}^{L}\Tr\left[Y_{pp}\mathrm{diag}(\mu_{p}^{(1)})\right]+2\sum_{p<q}\mathbf{OT}_{pq}(C_{pq}-Y_{pq}),
\]
 where $\mathbf{OT}_{pq}(A)$ is the optimal cost of the \emph{standard
}optimal transport problem with cost matrix $A$ and marginals $\mu_{p}^{(1)},\mu_{q}^{(1)}$.

Then compute 
\begin{eqnarray*}
f(Y+P_{i}v^{\top}) & = & -\sum_{p=1}^{L}\Tr\left[Y_{pp}\mathrm{diag}(\mu_{p}^{(1)})\right]-\Tr\left[\mathbf{1}_{2}v_{i}^{\top}\mathrm{diag}(\mu_{i}^{(1)})\right] + \Tr\left[\mathbf{1}_{2}v_{i+1}^{\top}\mathrm{diag}(\mu_{i+1}^{(1)})\right] \\
 &  & \ \ +\ 2\sum_{p<q,\,p\notin\{i,i+1\}}\mathbf{OT}_{pq}(C_{pq}-Y_{pq})\\
 &  & \ \ +\ 2\sum_{q=i+1}^{L}\mathbf{OT}_{iq}(C_{iq}-Y_{iq}-\mathbf{1}_{2}v_{q}^{\top})+2\sum_{q=i+2}^{L}\mathbf{OT}_{i+1,q}(C_{i+1,q}-Y_{i+1,q}+\mathbf{1}_{2}v_{q}^{\top}).
\end{eqnarray*}
 Now 
\[
\Tr\left[\mathbf{1}_{2}v_{i}^{\top}\mathrm{diag}(\mu_{i}^{(1)})\right]=v_{i}\cdot\mu_{i}^{(1)},\quad\Tr\left[\mathbf{1}_{2}v_{i+1}^{\top}\mathrm{diag}(\mu_{i+1}^{(1)})\right]=v_{i+1}\cdot\mu_{i+1}^{(1)},
\]
 and moreover it is not hard to see that 
\[
\mathbf{OT}_{pq}(A+\mathbf{1}_{2}x^{\top})=
\mathbf{OT}_{pq}(A) + 
x\cdot\mu_{q}^{(1)}
\]
 for any $A\in\R^{2\times2},x\in\R^{2}$, hence 
\begin{eqnarray*}
f(Y+P_{i}v^{\top}) & = & -\sum_{p=1}^{L}\Tr\left[Y_{pp}\mathrm{diag}(\mu_{p}^{(1)})\right]-v_{i}\cdot\mu_{i}^{(1)}+v_{i+1}\cdot\mu_{i+1}^{(1)}+2 \sum_{p<q}\mathbf{OT}_{pq}(C_{pq}-Y_{pq})\\
 &  & \ \ -\ 2\sum_{q=i+1}^{L}v_{q}\cdot\mu_{q}^{(1)}+2\sum_{q=i+2}^{L}v_{q}\cdot\mu_{q}^{(1)}\\
 & = & f(Y)-v_{i}\cdot\mu_{i}^{(1)}-v_{i+1}\cdot\mu_{i+1}^{(1)}.
\end{eqnarray*}
 Similarly
\begin{eqnarray*}
f(Y+vP_{i}^{\top}) & = & -\sum_{p=1}^{L}\Tr\left[Y_{pp}\mathrm{diag}(\mu_{p}^{(1)})\right]-\Tr\left[v_{i}\mathbf{1}_{2}^{\top}\mathrm{diag}(\mu_{i}^{(1)})\right]+\Tr\left[v_{i+1}\mathbf{1}_{2}^{\top}\mathrm{diag}(\mu_{i+1}^{(1)})\right]\\
 &  & \ \ +\ 2\sum_{p<q,\,q\notin\{i,i+1\}}\mathbf{OT}_{pq}(C_{pq}-Y_{pq})\\
 &  & \ \ +\ 2\sum_{p=1}^{i-1}\mathbf{OT}_{pi}(C_{pi}-Y_{pi}-v_{p}\mathbf{1}_{2}^{\top})+2\sum_{p=1}^{i}\mathbf{OT}_{p,i+1}(C_{p,i+1}-Y_{p,i+1}+v_{p}\mathbf{1}_{2}^{\top})\\
 & = & f(Y)+v_{i}\cdot\mu_{i}^{(1)}+v_{i+1}\cdot\mu_{i+1}^{(1)}.
\end{eqnarray*}
 Since the identities 
\[
f(Y+P_{i}v^{\top})=f(Y)-v_{i}\cdot\mu_{i}^{(1)}-v_{i+1}\cdot\mu_{i+1}^{(1)},\quad f(Y+vP_{i}^{\top})=f(Y)+v_{i}\cdot\mu_{i}^{(1)}+v_{i+1}\cdot\mu_{i+1}^{(1)}
\]
 hold for arbitrary $Y$, the claim Eq.~\eqref{eqn:fClaim} is proven.
 
 Then from Eq.~\eqref{eqn:fClaim} it follows that 
 \begin{equation}
 f(Y) = f(Y + P B + B^\top P^\top)
 \label{eqn:fClaimImp}
 \end{equation}
 for arbitrary $B \in \R^{(L-1)\times(2L)}$.
 
Now let $Q \in \R^{ (2L)\times(L+1)}$ be defined by 
\begin{equation*}
Q = \begin{bmatrix} w_1 & 0 & \cdots & 0 & w_2\\ 0 & w_1 &  & \vdots & \vdots \\ \vdots &  & \ddots  & & \\ 0 & \cdots &   & w_1 & w_2 \end{bmatrix},\quad w_1=\frac{1}{2}\begin{bmatrix} 1 \\ -1 \end{bmatrix},\quad w_2 = \frac{1}{2}\mathbf{1}_2,
\end{equation*}
and observe that $Q$ is chosen so that each column of $Q$ is orthogonal to each column of $P$. Moreover $P$ and $Q$ both 
have full rank, so it follows that 
$R := [Q, P]$ is invertible. 

Then for fixed $Y$, consider 
\[
\hat{Y}=R^{\top}YR=\left(\begin{array}{cc}
Q^{\top}YQ & Q^{\top}YP\\
P^{\top}YQ & P^{\top}YP
\end{array}\right).
\]
 We aim to choose $B$ such that 
\[
R^{\top}(PB+B^{\top}P^{\top})R=\left(\begin{array}{cc}
0 & 0\\
P^{\top}PBQ & P^{\top}PBP
\end{array}\right)+\left(\begin{array}{cc}
0 & Q^{\top}B^{\top}P^{\top}P\\
0 & P^{\top}B^{\top}P^{\top}P
\end{array}\right)
\]
 cancels $\hat{Y}$ on all but the top-left block. Using $Q^{\top}P=0$
(and $P^{\top}Q=0$), one can readily check that such a choice is
given by 
\[
-B=(P^{\top}P)^{-1}\hat{Y}_{21}(Q^{\top}Q)^{-1}Q^{\top}+\frac{1}{2}(P^{\top}P)^{-1}\hat{Y}_{22}(P^{\top}P)^{-1}P^{\top}.
\]
 By the identity \eqref{eqn:fClaimImp}, it follows that we can further restrict
the feasible set by intersecting with 
\begin{equation}
S'=\left\{ Y\,:\,R^{\top}YR=\left(\begin{array}{cc}
* & 0\\
0 & 0
\end{array}\right)\succeq0,\ f(Y) \geq d_0 \right\} .
\label{eqn:Sprime}
\end{equation}
In fact $S'$ is compact, and the proof is complete pending the proof of this claim, to which we now turn.

Observe that for 
$(Y, \{\phi_{pq},\psi_{pq}\}_{p<q})$ feasible, we may multiply 
Eq. \eqref{eqn:sdp2marDualPotentials} from the left by $\big(\mu_p^{(1)}\big)^\top$ and 
from the right by $\mu_q^{(1)}$ to obtain
\begin{eqnarray*}
\phi_{pq} \cdot \mu_p^{(1)} + \psi_{pq} \cdot \mu_q^{(1)} &\leq  & \big(\mu_p^{(1)}\big)^\top [C_{pq} - Y_{pq}] \big(\mu_q^{(1)}\big) \\
& = & \big(\mu_q^{(1)}\big)^\top [C_{pq} - Y_{pq}]^\top \big(\mu_p^{(1)}\big)\\
& = &  \Tr \left( [C_{pq} - Y_{pq} ]^\top\big(\mu_p^{(1)}\big) \big(\mu_q^{(1)}\big)^\top \right).
\end{eqnarray*}
By substituting this inequality into the objective function $F(Y, \{\phi_{pq},\psi_{pq}\}_{p<q})$ as defined in 
\eqref{eqn:sdp2marDual}, we see that 
\[
F(Y, \{\phi_{pq},\psi_{pq}\}_{p<q}) \leq \Tr (CM) - \Tr(YM).
\]
for $(Y, \{\phi_{pq},\psi_{pq}\}_{p<q})$ feasible, where
\begin{equation*}
M_{pq} :=\begin{cases}
\mathrm{\mathrm{diag}}\big(\mu_{p}^{(1)}\big), & p=q\\
\big(\mu_p^{(1)}\big) \big(\mu_q^{(1)}\big)^\top , & p\neq q.
\end{cases}
\end{equation*}
It follows then that 
\[
f(Y) \leq \Tr(CM) - \Tr(YM).
\]
In fact $M$ can be written $M = Q \widetilde{M} Q^\top$, where $\widetilde{M} \succ 0$. This can be verified directly by taking 
 \begin{equation*}
    \widetilde{M} =\begin{bmatrix} \widetilde \rho_1  \\ \vdots \\ \widetilde \rho_L \\
      1\end{bmatrix} \begin{bmatrix} \widetilde \rho_1  & \cdots & \widetilde \rho_L
        & 1\end{bmatrix} + \text{diag}\left(\begin{bmatrix}1-\widetilde \rho_1^2 &
          \cdots & 1-\widetilde{\rho}_L^2 & 0 \end{bmatrix}\right),
   \end{equation*} 
with 
   \begin{equation*}
     \widetilde \rho_p = 1-2\rho_p,\quad p=1,\ldots,L.
   \end{equation*}
Note that $\widetilde{M} \succ 0$ by the assumption \eqref{eqn:rho_assumption_0}.
Hence 
\[
f(Y) \leq \Tr(CM) - \Tr(Q^\top Y Q \widetilde{M}).
\]
Now since $\widetilde{M} \succ 0$, there exists a scalar $K > 0$ such that if $Y \succeq 0$ and $Q^\top Y Q \not\preceq K$, 
then $f(Y) < d_0$. But $Q^\top Y Q$ is the upper-left block of $R^\top Y R$, so it follows from the definition \eqref{eqn:Sprime} of $S'$ that 
that
\[
S' \subset \left\{ Y\, :\, R^\top Y R = \left(\begin{array}{cc}
A & 0\\
0 & 0
\end{array}\right),\, 0 \preceq A \preceq K \right\},
\]
from which it follows that $S'$ is compact, and the proof is complete.
\end{proof}

\begin{remark}
Note that the proof of Theorem \ref{thm:strongDuality} guarantees that the domain of the dual problem \eqref{eqn:sdp2marDual} 
can be restricted to $Y$ of the form $Y = Q \widetilde{Y} Q^\top$, yielding a `reduced' dual problem in which $\widetilde{Y}$ replaces 
$Y$ as an optimization variable. In fact, one can also verify directly that any $M$ feasible for 
the primal problem \eqref{eqn:sdp2mar} satisfies $MP = 0$, hence the domain of the primal problem can be restricted to $M$ 
of the form $M = Q \widetilde{M} Q^\top$, likewise yielding a reduced primal problem.

But despite this apparent symmetry, the latter observation need not imply the former in a 
more general SDP setting, and the arguments given in the proof of Theorem \ref{thm:strongDuality}, which use more of the 
specific structure of our problem, do appear to be necessary to the proof of dual attainment for this problem.

Moreover, observe with caution that the dual of such a reduced primal problem is \emph{not} the reduced dual problem!
\end{remark}

\section{Tighter lower bound via 3-marginals}
\label{sec:three marginal convex}
In this section, we further tighten the convex relaxation proposed in Section \ref{sec:two marginal convex} with a formulation that additionally involves the 3-marginals. 

One defines the 3-marginals $\mu^{(3)}_{pqr}$ (for $p,q,r,$ distinct) induced by a probability measure $\mu$ on $\{0,1\}^L$ via 
\begin{equation}
\mu^{(3)}_{pqr}(s_{p},s_{q},s_r) :=
\sum_{s_1,\ldots,s_L\backslash\{s_{p},s_{q},s_r\}} \mu(s_1,\ldots,s_L).
  \label{eqn:threemarginal}
\end{equation}
There is no 3-marginal analog known to us of the semidefinite constraint that can be enforced using the 2-marginals. However, 
we can nonetheless use the 3-marginals to enforce additional necessary \emph{local consistency} constraints. Indeed, the 2-marginals 
can themselves be written in terms of the 3-marginals via 
\begin{equation}
\mu^{(2)}_{pq}(s_{p},s_{q}) = \sum_{s_r} \mu^{(3)}_{pqr}(s_{p},s_{q},s_r).
\end{equation}

Accordingly, we will include $K = \{K_{pqr}\}$ for distinct $p,q,r$ as optimization variables for the 3-marginals. 
Note that based on Eq.~\eqref{eqn:threemarginal} we can enforce that $K$ is \emph{symmetric}, 
by which we mean that 
\[
K_{pqr} (s_p, s_q, s_r) = K_{\sigma(p) \sigma(q) \sigma(r)}(s_{\sigma(p)},s_{\sigma(q)},s_{\sigma(r)})
\]
for any permutation $\sigma$ on the letters $\{p,q,r\}$. If we were to extend $K_{pqr}$ by zeros to $p,q,r$ not distinct, 
then we could think of $K\in \R^{(2L)\times(2L)\times(2L)}$ as a symmetric 3-tensor, with $(p,q,r)$-th $2\times2 \times 2$ 
block given by $K_{pqr}$.
 In principle the imposition of symmetry removes some redundancy in the specification of $K$.

Then we arrive at the following \emph{3-marginal SDP}: 
\begin{eqnarray}
&\underset{M \in \R^{(2L)\times(2L)},\, K\in \R^{(2L)\times(2L)\times(2L)}  }{\mathrm{minimize}} \ \ & \Tr(CM) \label{eqn:sdp3mar} \\
&\text{subject to}\ \ & \ M \succeq 0, \cr
&\ & \ M_{pq} \geq 0 \,\text{ for } p\neq q , \cr
&\ & \ M_{pq} \mathbf{1}_2 = \mu_p^{(1)} \,\text{ for } p\neq q , \cr
&\ & \ 
M_{pp} = \mathrm{diag}(\mu_p^{(1)}) \,\text{ for all }\, p, \nonumber \\
&\ & \ K \geq 0, \ K \text{ symmetric}, \cr 
&\ & \  M_{pq}(s_p,s_q) = \sum_{s_r} K_{pqr} (s_p,s_q, s_r) \,\text{ for } p, q, r \,\text{distinct}. \nonumber
\end{eqnarray}
Note that the blocks $K_{pqr}$ for $p,q,r$ not distinct are superfluous and can be discarded in an efficient optimization.

For simplicity, we omit discussion of the duality of 
\eqref{eqn:sdp3mar}. Since only linear constraints have been added, most of the interesting features 
from the mathematical viewpoint have already been discussed above. Indeed, as in Section~\ref{sec:dual}, we may derive the dual of the 3-marginal problem~\eqref{eqn:sdp3mar}, and we may certify as in Section~\ref{sec:strong} that the 3-marginal problem satisfies strong duality and dual attainment.

\section{General MMOT with pairwise cost}
\label{sec:general OT}
As has been suggested both explicitly and via the notation, almost all of our discussion of relaxation methods for MMOT can be applied to general MMOT problems with pairwise cost functions. The main caveat is that specific references to the fact that the 1-marginal state space 
has two elements should be suitably generalized. For clarity, we now recapitulate our methods for the general MMOT problem with pairwise cost. The reader interested in general MMOT should still see the earlier sections for derivations, discussions, and proofs. Here we only summarize the methods.

We will consider a problem with $L$ marginals, written $\mu_p^{(1)}$ for $p=1,\ldots,L$. These quantities are fixed in advance 
and never varied in the following discussion.
We let $N_p$ be the size of the state 
space of the $p$-th marginal, so $\mu_p^{(1)}$ is a probability vector of length $N_p$. Note that the marginals need not all have the same state space, i.e., $N_p$ can depend on $p$. We write the $p$-th state space as $\mathcal{X}_p := \{1,\ldots,N_p\}$. Then the joint 
state space is given by $\mathcal{X} := \prod_{p=1}^L \mathcal{X}_p$, and we write $\mathrm{Pr}_p$ for the $p$-th projection $\mathcal{X} \ra \mathcal{X}_p$.
Suppose that we are given a pairwise cost function $C_{pq} \in \R^{N_p \times N_q}$ for each pair $p\neq q$ of marginals. 
(Without loss of generality we assume $C_{pp} = 0$.) Then we consider the problem 
\begin{equation}
\min_{\mu \in \mathcal{P}(\mathcal{X})} \sum_{(s_1,\ldots,s_L) \in \mathcal{X}} \sum_{p,q=1}^L C_{pq}(s_p,s_q) \mu(s_1,\ldots,s_L),\quad \text{s.t.}\ \ (\mathrm{Pr}_p) \# \mu = \mu_p^{(1)}, \ p=1,\ldots,L.
\end{equation}
Here $\mu : \mathcal{X} \ra \R$ can be thought of as an $L$-tensor whose $p$-th index ranges from $1,\ldots,N_p$. Again, the objective function of such a MMOT problem can be rephrased in terms of the 2-marginals:
\begin{equation}
\min_{\mu \in \mathcal{P}(\mathcal{X})} \sum_{p \neq q}^L \Tr(C_{pq}\mu^{(2)}_{pq}),\quad \text{s.t.}\ \ (\mathrm{Pr}_p) \# \mu = \mu_p^{(1)}, \ p=1,\ldots,L,
\end{equation}
where the 2-marginals $\mu^{(2)}_{pq}$ are here implicitly defined in terms of the optimization variable $\mu$.

Then we introduce the \emph{2-marginal primal SDP}
\begin{eqnarray}
& \underset{M \in \R^{N_{\mathrm{tot}}\times N_{\mathrm{tot}} }}{\mathrm{minimize}} \ \ & \Tr(CM) \label{eqn:sdp2marGeneral} \\
&\text{subject to}\ \ & \ M \succeq 0, \cr
&\ & \ M_{pq} \geq 0 \,\text{ for all }\, p,q=1,\ldots,L\ (p\neq q), \cr
&\ & \ M_{pq} \mathbf{1}_{N_q} = \mu_p^{(1)} \,\text{ for all }\, p,q=1,\ldots,L\ (p\neq q), \cr
&\ & \ 
M_{pp} = \mathrm{diag}(\mu_p^{(1)}) \,\text{ for all }\, p=1,\ldots,L. \nonumber
\end{eqnarray}
Here $N_{\mathrm{tot}} := \sum_{p=1}^L N_p$ and $\mathbf{1}_k$ denotes the vector of ones of length $k$. The dual of 
\eqref{eqn:sdp2marGeneral} is given by 
\begin{eqnarray}
& \underset{ Y, \, \{\phi_{pq}, \psi_{pq} \}_{p< q}}{\mathrm{maximize}} \ \ & 
2 \sum_{p<q} \left( \phi_{pq} \cdot \mu_p^{(1)} + \psi_{pq} \cdot \mu_q^{(1)} \right) 
- \sum_{p,s} Y_{pp}(s,s)\mu_p^{(1)}(s)
\label{eqn:sdp2marDualGeneral} \\
&\text{subject to}\ \ & \ Y \succeq 0, \nonumber  \\
&\ & \ \phi_{pq}\mathbf{1}_{N_q}^\top + \mathbf{1}_{N_p} \psi_{pq}^\top \leq C_{pq} - Y_{pq}
 \,\text{ for }\, p<q. \nonumber
\end{eqnarray}
In \eqref{eqn:sdp2marDualGeneral} is it understood that $Y\in \R^{N_{\mathrm{tot}}\times N_{\mathrm{tot}}}$ and 
moreover $\phi_{pq} \in \R^{N_p}$, $\psi_{pq} \in \R^{N_q}$.

By generalizing the discussion of Theorem \ref{thm:strongDuality}, we have strong duality for the 2-marginal SDP, hence 
the optimal values of \eqref{eqn:sdp2marGeneral} and \eqref{eqn:sdp2marDualGeneral} are equal, and moreover the dual 
problem admits a maximizer. (The primal problem admits a maximizer trivially because the feasible set is compact.)

Finally, we turn to the \emph{3-marginal primal SDP}
\begin{eqnarray}
&\underset{M \in \R^{N_{\mathrm{tot}}\times N_{\mathrm{tot}}},\, K\in \R^{N_{\mathrm{tot}} \times N_{\mathrm{tot}} \times N_{\mathrm{tot}}}  }{\mathrm{minimize}} \ \ & \Tr(CM) \label{eqn:sdp3marGeneral} \\
&\text{subject to}\ \ & \ M \succeq 0, \cr
&\ & \ M_{pq} \geq 0 \,\text{ for } p\neq q , \cr
&\ & \ M_{pq} \mathbf{1}_{N_q} = \mu_p^{(1)} \,\text{ for } p\neq q , \cr
&\ & \ 
M_{pp} = \mathrm{diag}(\mu_p^{(1)}) \,\text{ for all }\, p, \nonumber \\
&\ & \ K \geq 0, \ K \text{ symmetric}, \cr 
&\ & \  M_{pq}(s_p,s_q) = \sum_{s_r} K_{pqr} (s_p,s_q, s_r) \,\text{ for } p, q, r \,\text{distinct}. \nonumber
\end{eqnarray}
For simplicity we omit the concrete formulation of the corresponding dual problem, but we note that strong duality 
and dual attainment can be proved by methods similar to those applied in the 2-marginal case.

\section{Numerical results}
\label{sec:numeric}
In this section, we numerically demonstrate the effectiveness of the proposed methods on solving for the ground state energy of Hubbard-type models. The Hubbard  model is a prototypical model for strongly-correlated quantum systems, and is considered to be of significant importance to model behaviors such as high temperature superconductivity \cite{raghu2010superconductivity}. To investigate the numerical performance of the SDP formulation, we consider one-dimensional spinless Hubbard models and two-dimensional spinful Hubbard models. 

\subsection{One-dimensional spinless model}

Here we consider a 1D spinless Hubbard-like model defined by the Hamiltonian
of Eq.~\eqref{eqn:hamiltonian}, in which we take 
\begin{equation}
  t_{pq} = \begin{cases} 1 & \text{if }\vert q-p \vert = 1,\\ 0 & \text{otherwise} \end{cases}
  \label{}
\end{equation}
and consider two different cases of $v$, with next-nearest neighbor (NN) interaction, 
\begin{equation}
  v_{pq} = \begin{cases} U/2 & \text{if }\vert q-p \vert = 1,\\ U/40 & \text{if }\vert q-p \vert = 2, \\
  0 & \text{otherwise} \end{cases}
  \label{2 hop}
\end{equation}
and next-next-nearest  neighbor interaction (NNNN)
\begin{equation}
\label{3 hop}
  v_{pq} = \begin{cases} U/2 & \text{if }\vert q-p \vert = 1,\\ U/20 & \text{if }\vert q-p \vert = 2, \\
  U/200 & \text{if }\vert q-p \vert = 3, \\ 0 & \text{otherwise.}  \end{cases}
\end{equation}
The reason why we omit the obvious scenario of the nearest neighbor (NN) interaction is that in such a case, we find that our convex relaxation becomes \textit{numerically exact} and hence we consider the case to be not representative. We do not have a proof yet to explain why our convex relaxation scheme can be numerically exact.

We will compare the Kohn-Sham SCE energies yielded by our methods with one another, as well 
as with the exact ground-state
energy~\eqref{eqn:groundstate}, which is computed via exact diagonalization (ED) in 
the \textsf{OpenFermion}~\cite{mcclean2017openfermion} software package. The MMOT 
problems arising in Kohn-Sham SCE and their SDP
relaxations are
solved in \textsf{MATLAB} with the \textsf{CVX} software
package~\cite{grant2008cvx}.

We refer to the exact self-consistent Kohn-Sham SCE solution obtained by solving
the original linear programming (LP) problem for MMOT as the `LP' solution. Hence 
 the tightness of the Kohn-Sham SCE lower bound \eqref{eqn:E_Kohn-Sham SCE}
\emph{itself} can be evaluated by comparing 
the exact energy with the LP energy, while the tightness of our 
SDP \emph{relaxations} of the relevant MMOT problems 
(which, in turn, yield lower bounds for the Kohn-Sham SCE energy) 
can be evaluated by comparing the LP energy with the 2- and 3-marginal 
SDP energies. We refer to these two sources of error, respectively, as the 
`Kohn-Sham SCE model error' and the `error due to relaxation.'

In Figs. \ref{fig:hubbard_1D_2hop}(a) and \ref{fig:hubbard_1D_3hop}(a), we plot $E/U$ with respect to $U$ for $v$ as in Eqs. \eqref{2 hop} and \eqref{3 hop}, respectively. In these experiments, $L=14$ and $N = 9$. The energy differences of the Kohn-Sham SCE solutions from the exact energy are plotted in Figs. \ref{fig:hubbard_1D_2hop}(b) and \ref{fig:hubbard_1D_3hop}(b). It is confirmed numerically 
that the LP energy lower-bounds the exact energy, and in turn the SDP energies lower-bound the LP energy. 
While the 3-marginal SDP lower bound is noticeably tighter than the 2-marginal SDP lower bound, the error due to relaxation is dominated by the Kohn-Sham SCE model error in both cases.

Since the effective potential is of interest in Kohn-Sham DFT, in Fig. \ref{fig:dual potential} we plot the SCE potential \eqref{eqn:SCE_potential} at self-consistency in the case of $v$ as in Eq. \eqref{3 hop}. It can be seen that the 3-marginal SDP performs better than the 2-marginal SDP in this regard, as one might expect. (However, note carefully that although it is guaranteed \emph{a priori} that the 3-marginal SDP provides a lower bound on the \emph{energy} that is at least as tight as that of the 2-marginal SDP, no such comparison is theoretically guaranteed in advance for the effective potential.)

To study the scaling of energy in the thermodynamic limit $L\ra\infty$, in Fig \ref{fig:hubbard_1D_sweepL}(a), we plot $E/U$ as a function of $L$ by fixing $U=5$ and a filling factor of $N/L = 2/3$. In Fig \ref{fig:hubbard_1D_sweepL}(b), we plot the total runtime of our methods on a MacBook Pro with a 2.3GHz Core I5 CPU and 16GB of memory.

\begin{figure}[h]
  \begin{center}
    \subfloat[]{\includegraphics[width=6cm]{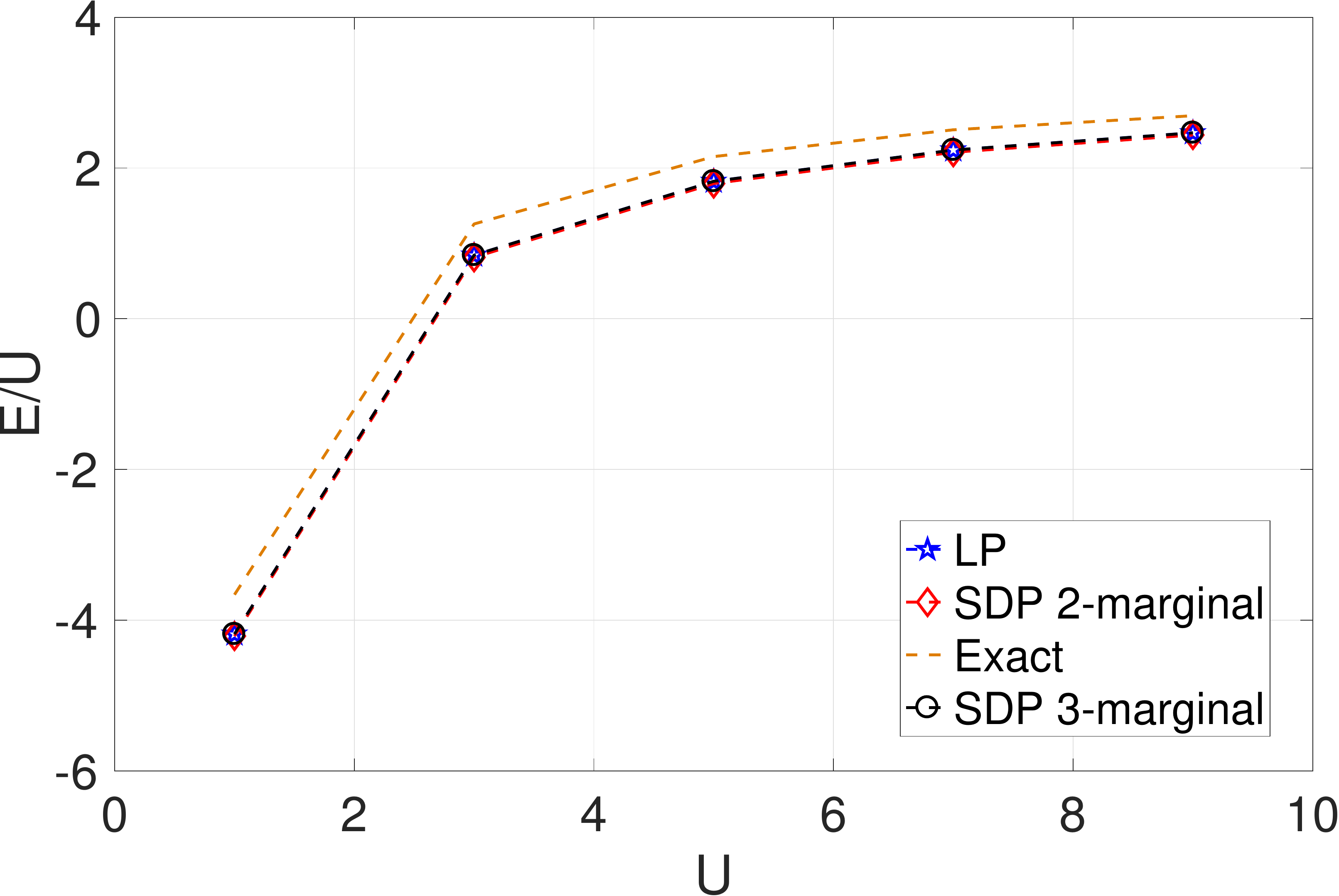}}
    \subfloat[]{\includegraphics[width=6cm]{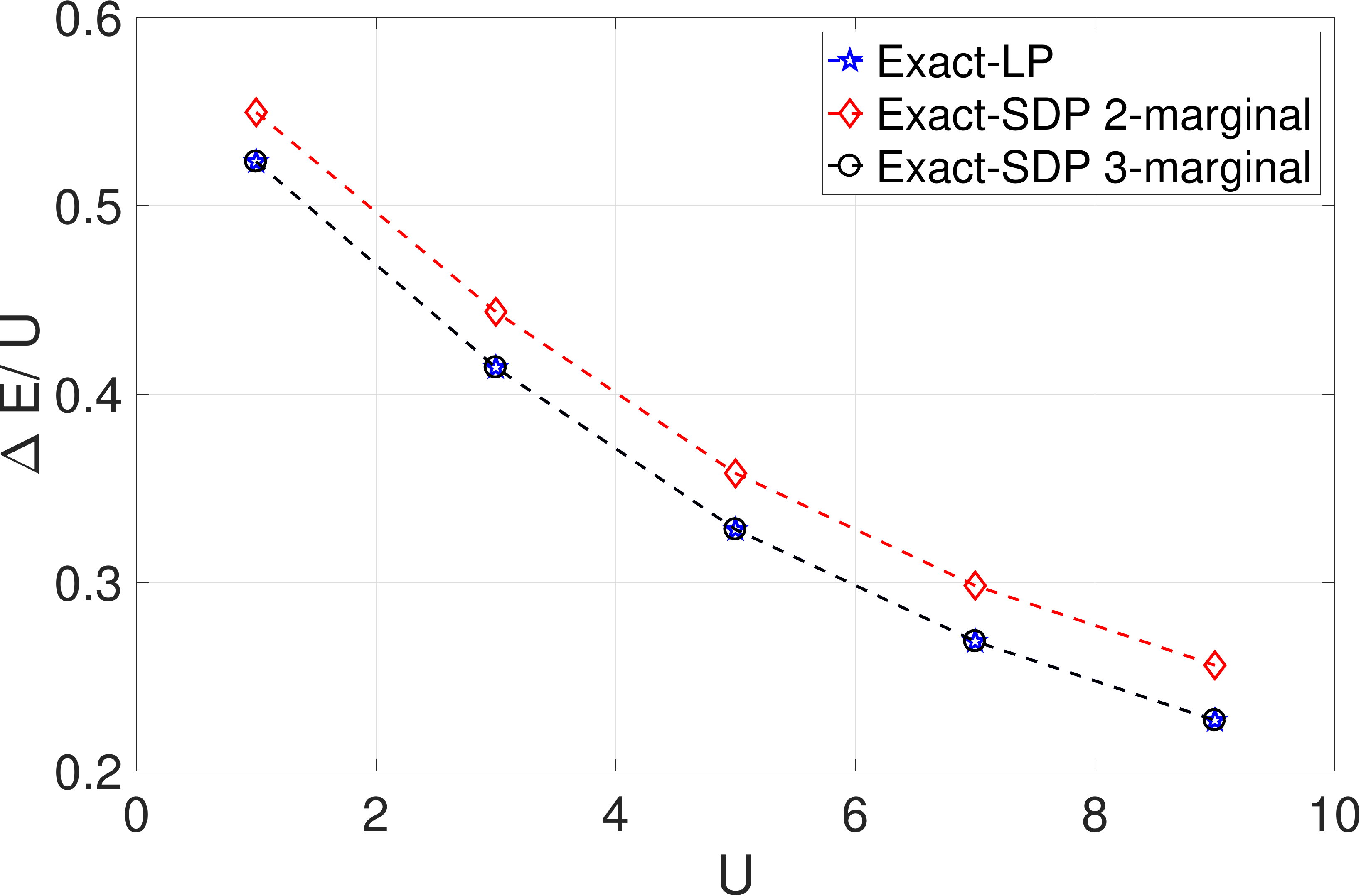}}
  \end{center}
  \caption{Spinless 1D fermionic lattice model with $v$ as in Eq. \eqref{2 hop}, $L=14$, $N = 9$. (a) $E/U$ as a function of $U$. (b) Difference between the exact energy and the Kohn-Sham SCE energies obtained from the unrelaxed LP and the SDP relaxations.}
  \label{fig:hubbard_1D_2hop}
\end{figure}

\begin{figure}[h]
  \begin{center}
    \subfloat[]{\includegraphics[width=6cm]{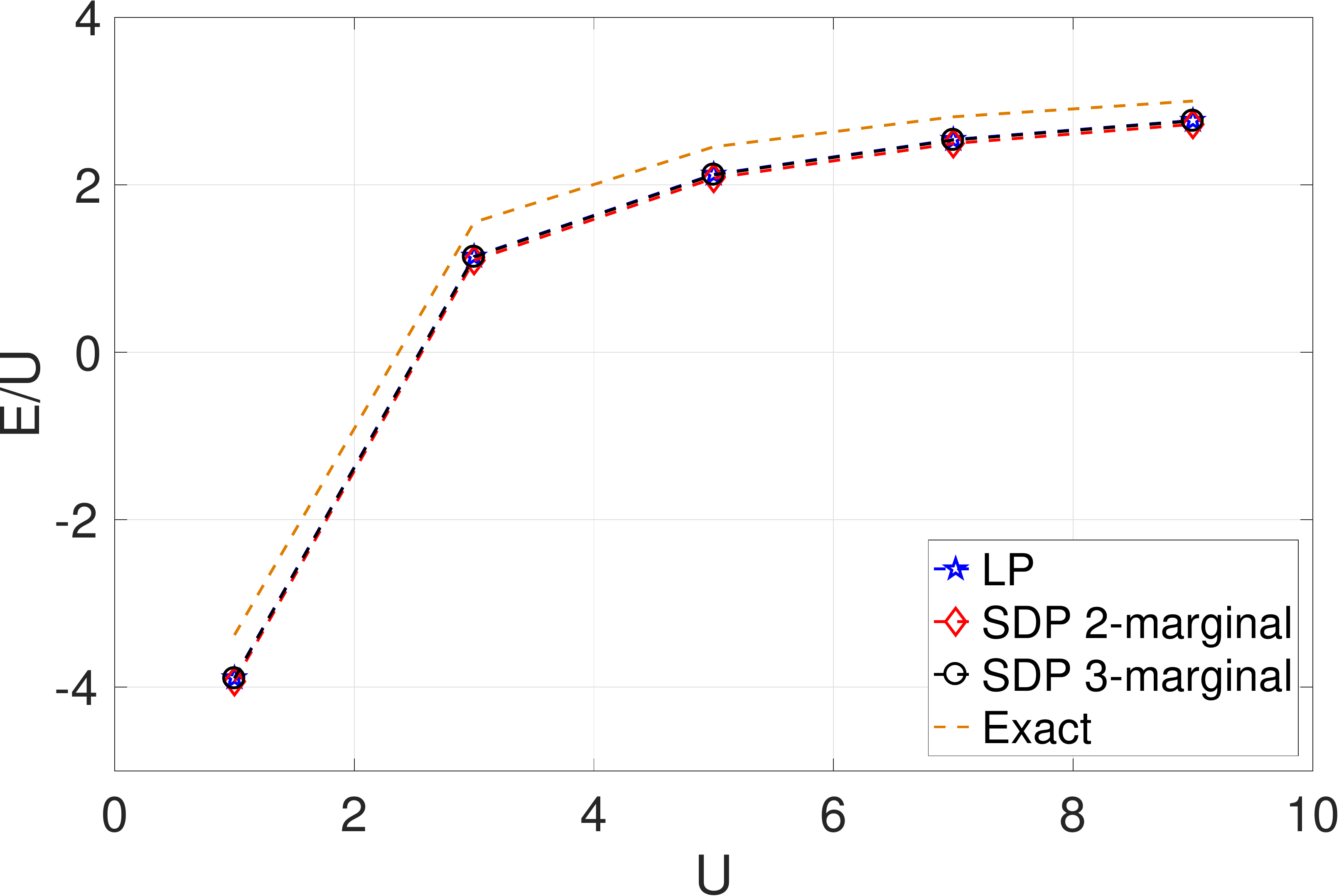}}
    \subfloat[]{\includegraphics[width=6cm]{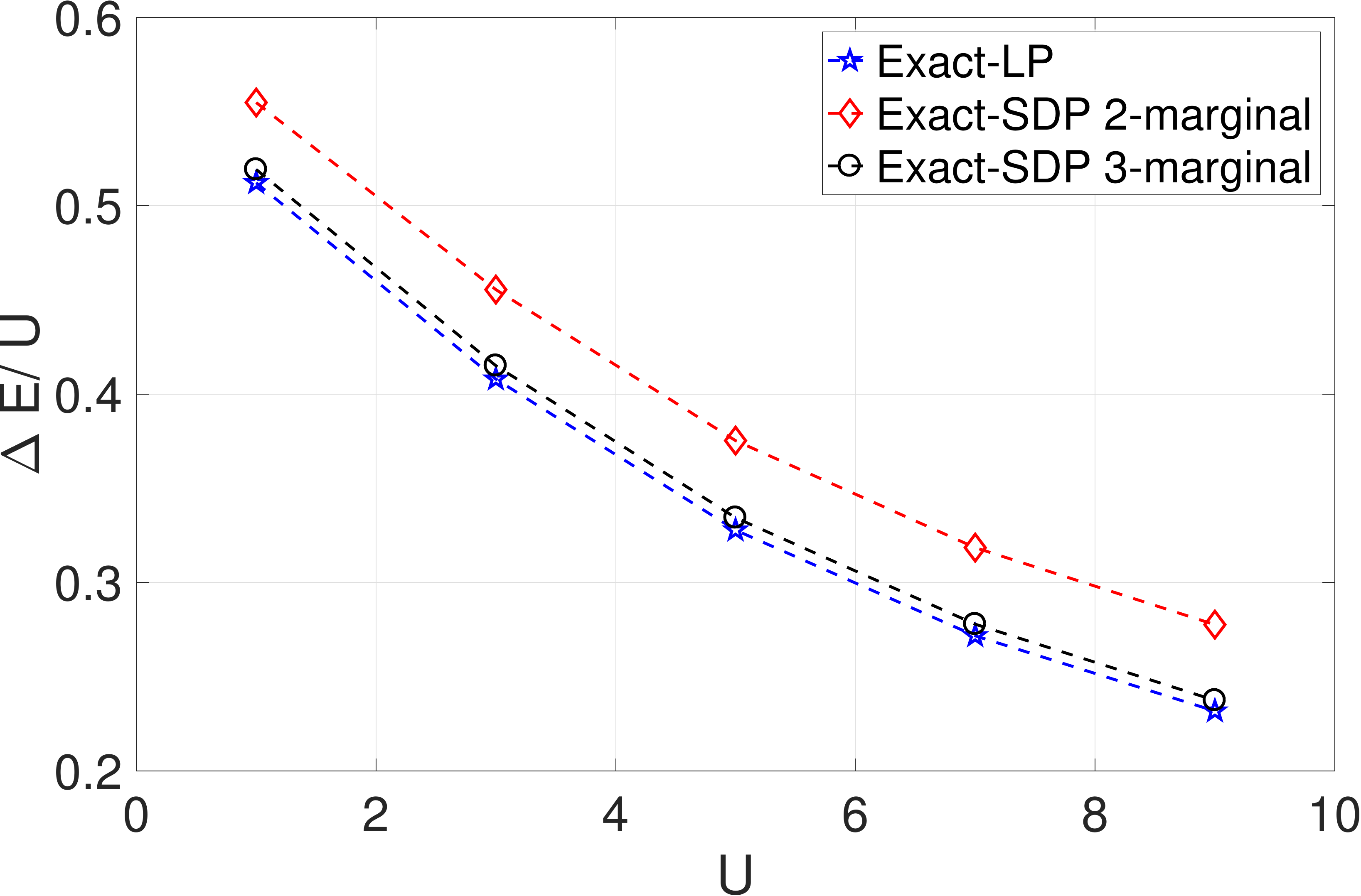}}
  \end{center}
  \caption{Spinless 1D fermionic lattice model with $v$ as in Eq. \eqref{3 hop}, $L=14$, $N = 9$. (a) $E/U$ as a function of $U$. (b) Difference between the exact energy and the Kohn-Sham SCE energies obtained from the unrelaxed LP and the SDP relaxations.}
  \label{fig:hubbard_1D_3hop}
\end{figure}

\begin{figure}[h]
  \begin{center}
    \includegraphics[width=6cm]{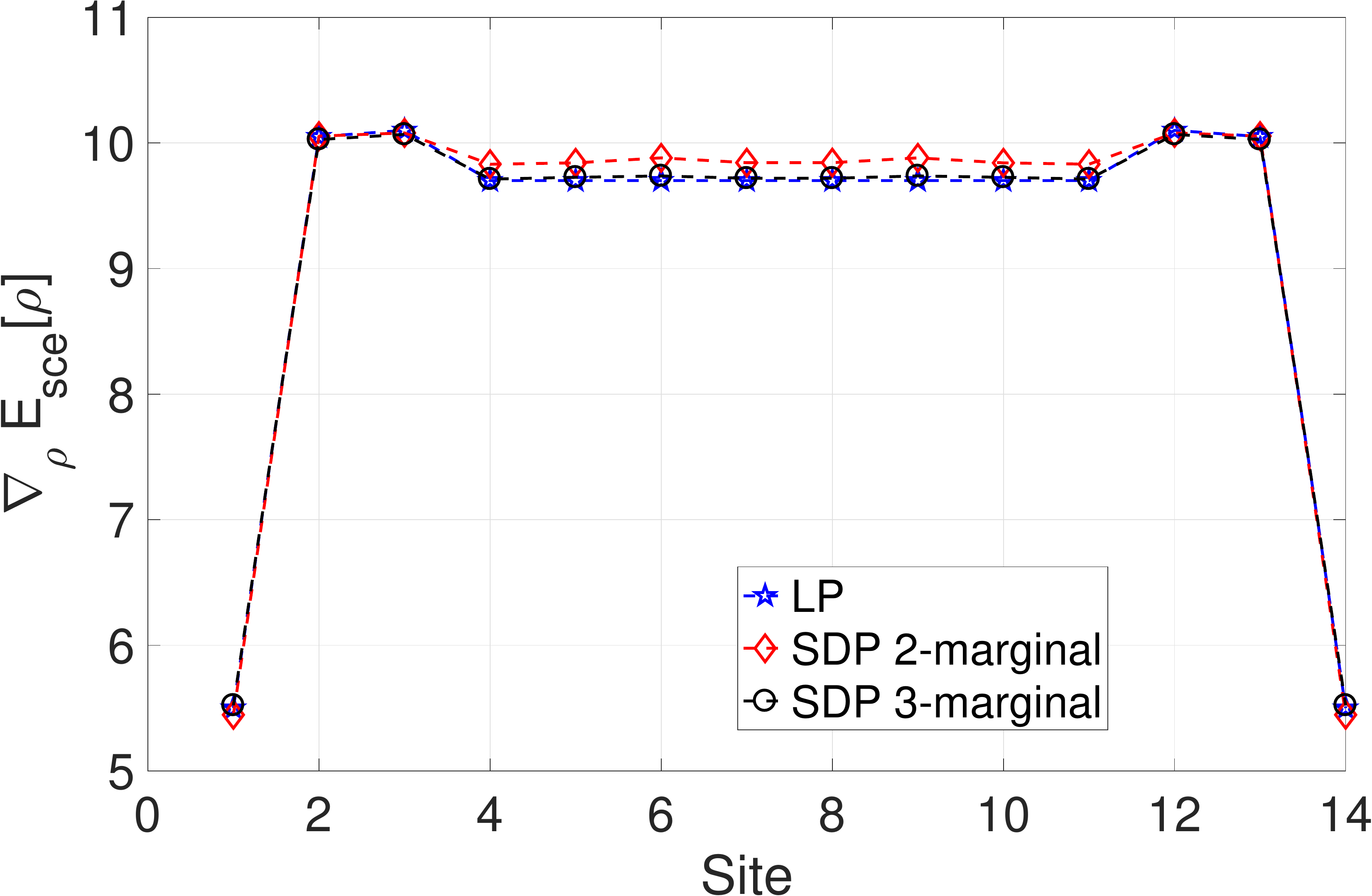}
  \end{center}
  \caption{The effective potential for the spinless 1D fermionic lattice model with $v$ as in Eq. \eqref{3 hop}, $U=5$, $L=14$, $N=9$. The relative $\ell^2$ errors for the 2- and 3-marginal formulations (compared to the unrelaxed LP formulation) are $1.2 \times 10^{-2}$ and $2.7 \times 10^{-3}$, respectively. }
  \label{fig:dual potential}
\end{figure}

\begin{figure}[h]
  \begin{center}
    \subfloat[]{\includegraphics[width=5.8cm]{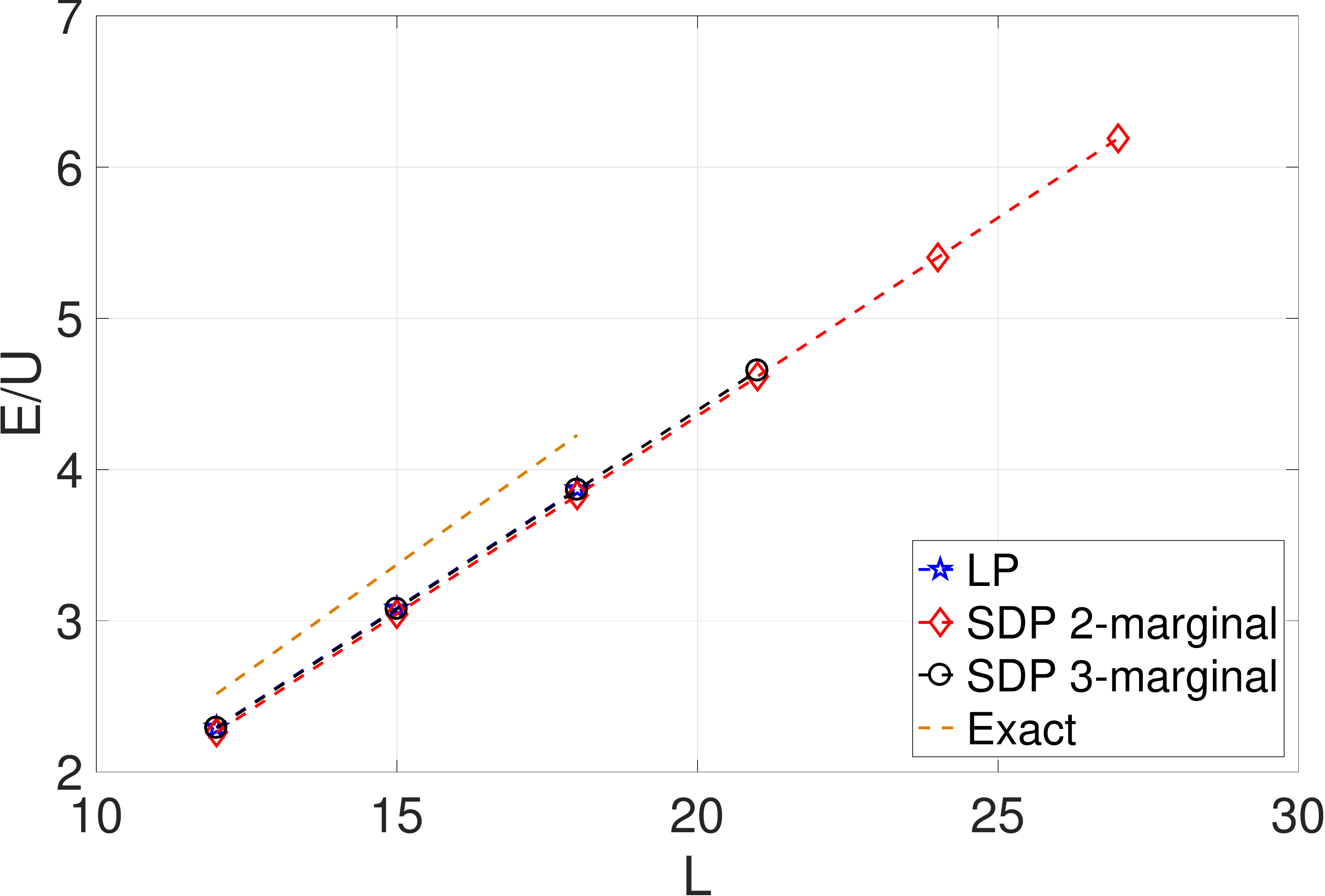}}
    \subfloat[]{\includegraphics[width=6cm]{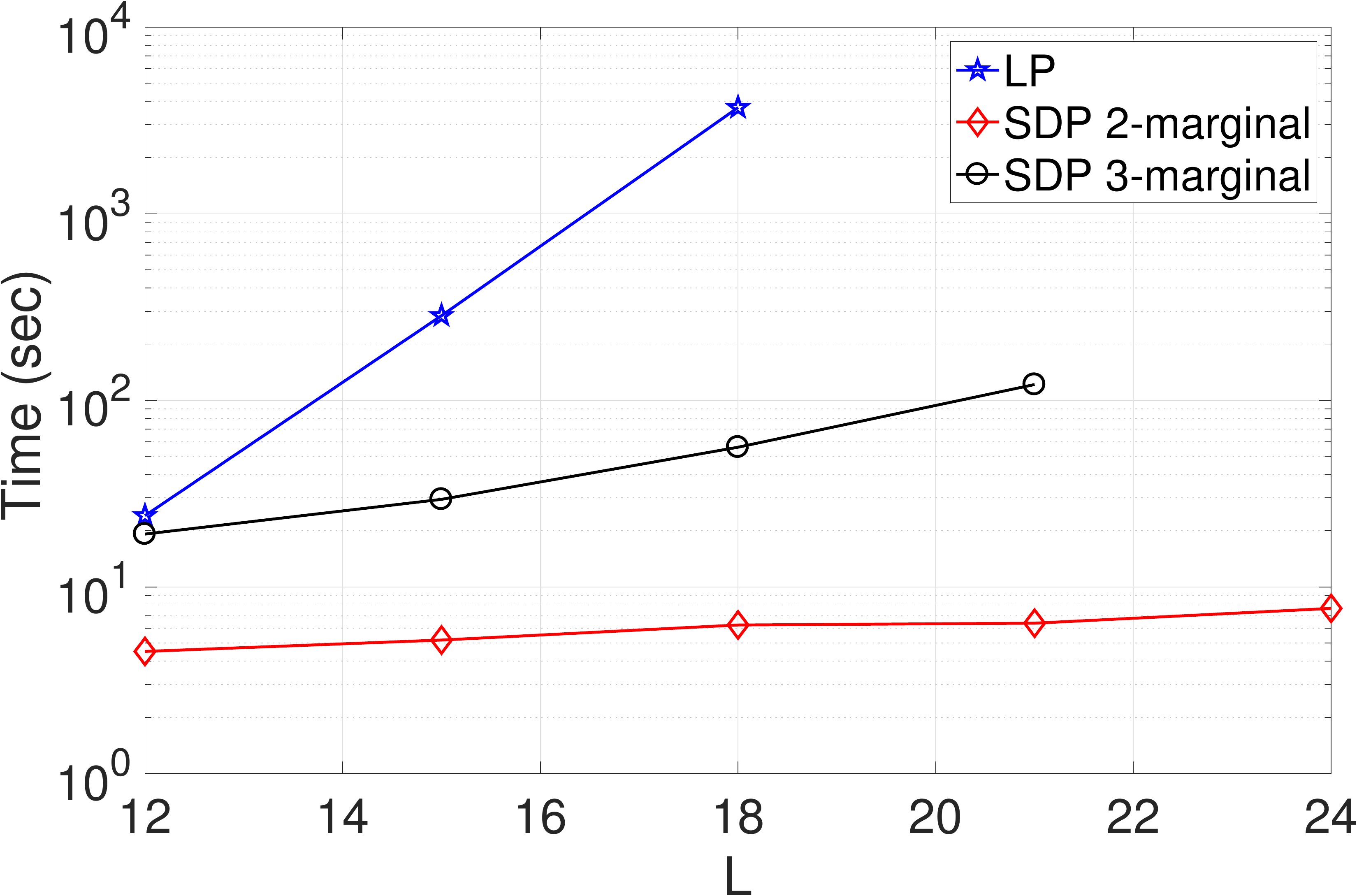}}
  \end{center}
  \caption{Spinless 1D fermionic lattice model with $v$ as in Eq. \eqref{3 hop}, $U=5$, $N/L=2/3$. (a) $E/U$ as a function of $L$. (b) Running time as a function of $L$.}
  \label{fig:hubbard_1D_sweepL}
\end{figure}

\subsection{Two-dimensional spinful model}

We consider a 2D generalized Hubbard type model defined by the Hamiltonian 
\begin{equation}
\begin{split}
     \hat{H} =& -\sum_{i,j=1}^{L-1}
     \sum_{\sigma\in\{\uparrow,\downarrow\}} 
  \left(\hat{a}_{i+1,j;\sigma}^\dagger \hat{a}_{i,j;\sigma} +
  \hat{a}_{i,j+1;\sigma}^\dagger \hat{a}_{i,j;\sigma} + \text{h.c.}\right) \\
  &+ U \sum_{i,j=1}^{L}
  \hat{n}_{i,j;\uparrow} \hat{n}_{i,j;\downarrow} 
   + V \sum_{i,j=1}^{L-1}
  \left(\hat{n}_{i+1,j} \hat{n}_{i,j} + \hat{n}_{i,j+1}
  \hat{n}_{i,j}\right). 
\end{split}
  \label{eqn:generalizedHubbard}
\end{equation}
Here $\hat{n}_{i,j}:=\hat{n}_{i,j;\uparrow}+\hat{n}_{i,j;\downarrow}$. As discussed in section~\ref{sec:prelim}, although the creation and annihilation operators in 
Eq.~\eqref{eqn:generalizedHubbard} involve two spatial indices and one spin index, one may of course order the operators with a single index by defining 
\[
b_{(j-1)L+i}=a_{i,j;\uparrow}, \quad b_{(j-1)L+i+L^2}=a_{i,j;\downarrow}.
\]
The new creation operators are fixed as the Hermitian adjoints of these new annihilation operators. The term associated with $U$ is the on-site electron-electron interaction, while $V$ specifies the nearest-neighbor electron-electron interaction.  In the standard Hubbard model, we have $V=0$. (However, in the case $V=0$, the MMOT problem arising in the SCE framework becomes a trivial problem, since the interaction terms associated with different sites are decoupled.) Fig.~\ref{fig:hubbard_spinful_2D_Lx3Ly3Ne12} 
shows the energies for the generalized Hubbard model on a $3\times 3$
lattice, with $V=0.05\, U$ and $U$ ranging from $1.0$ to $19.0$.  The
number $N$ of electrons is set to be $12$. Here energies are obtained
from the exact solution, the exact Kohn-Sham SCE solution obtained by
linear programming (LP), and the approximate Kohn-Sham SCE solution
obtained via the 2-marginal SDP relaxation. We find that the Kohn-Sham
SCE formulation becomes asymptotically accurate when $U$ becomes large.
Furthermore, the error due to relaxation is much smaller than the
Kohn-Sham SCE model error.
Fig.~\ref{fig:hubbard_spinful_2D_Lx3Ly3Ne12}(b)  further shows that the
energy difference between the LP and 2-marginal SDP solutions is approximately constant with respect to the on-site interaction strength $U$. 
\begin{figure}
  \begin{center}
        \subfloat[]{\includegraphics[width=0.4\textwidth]{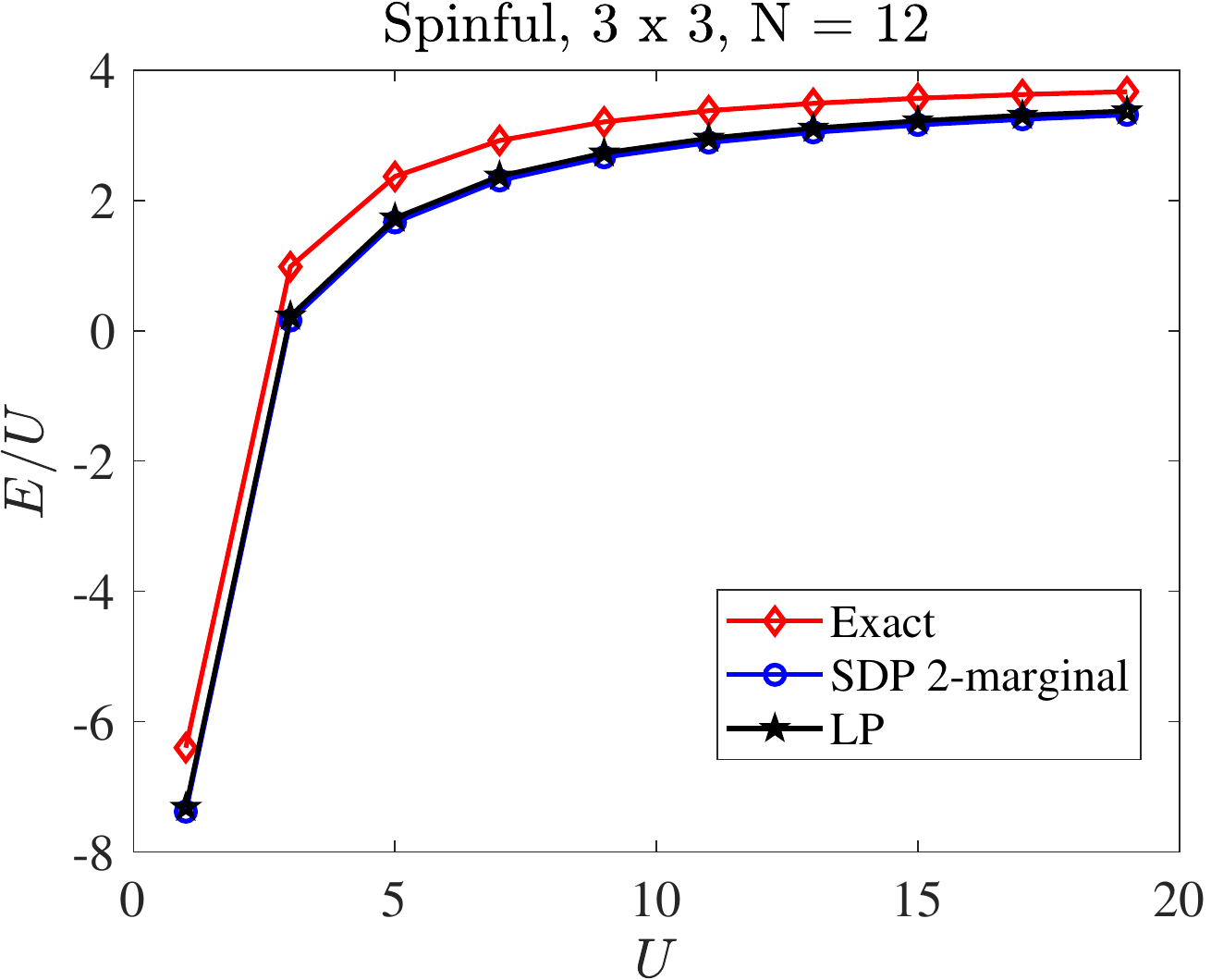}}
        \subfloat[]{\includegraphics[width=0.4\textwidth]{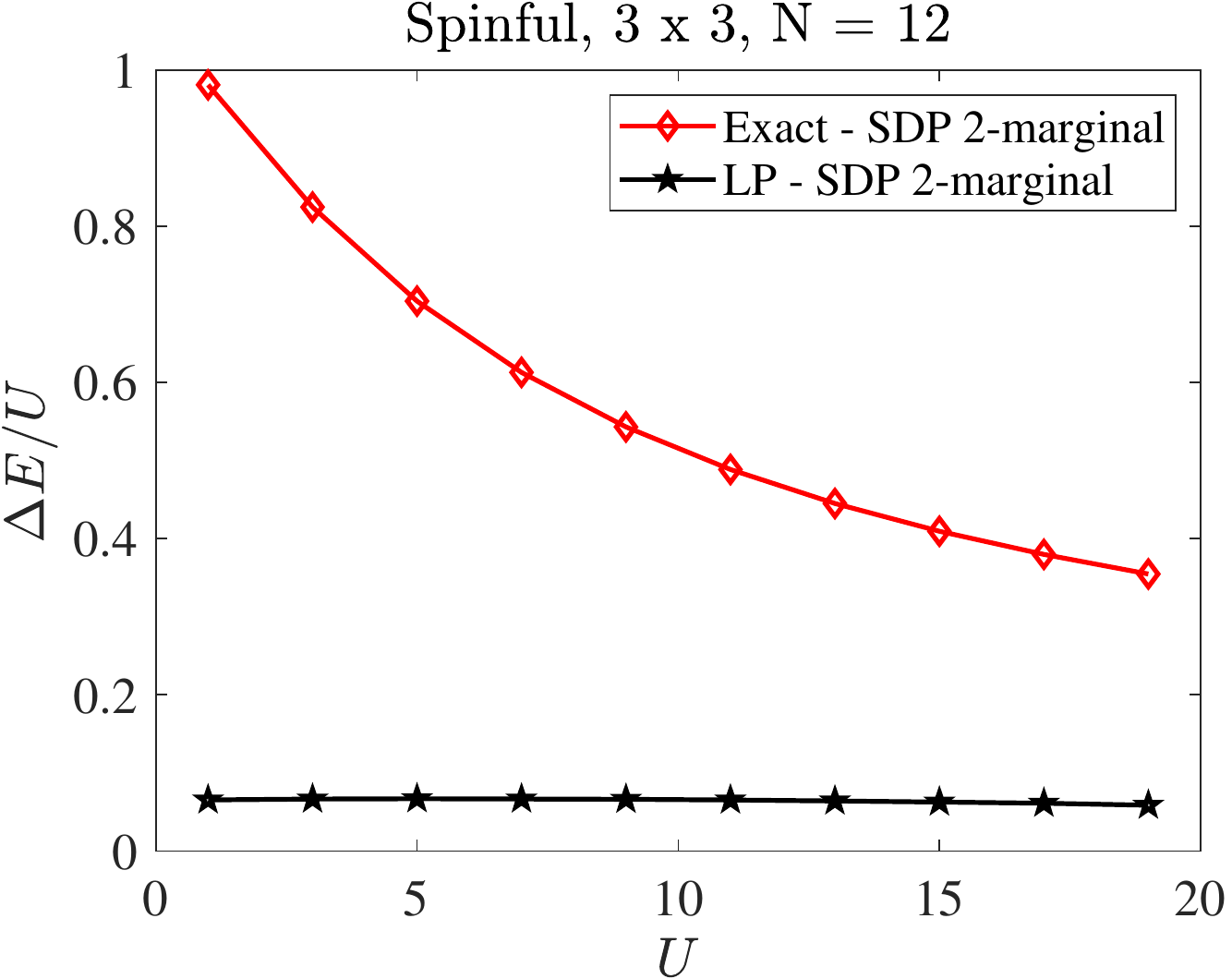}}
  \end{center}
  \caption{Spinful $3\times 3$ Hubbard model with $N=12$.}
  \label{fig:hubbard_spinful_2D_Lx3Ly3Ne12}
\end{figure}

\section{Conclusion}\label{sec:conclusion}

In this paper, we have considered the strictly correlated electron (SCE)
limit of a fermionic quantum many-body system in the second-quantized formalism. To the
extent of our knowledge, the setup of the SCE problem in this setting has not appeared in
the literature. Mathematically, the SCE limit requires the solution of a
multi-marginal optimal transport problem over certain classical
probability measures. We propose a relaxation that enforces constraints on 
the 2-marginals of these measures, and the relaxed problem can be solved
efficiently via semi-definite programming (SDP). We prove that the SDP
problem satisfies strong duality and moreover that the dual solution is attained, 
despite the fact that the primal problem does
not possess a strictly feasible point.  We consider a tighter relaxation involving 
the 3-marginals and discuss how our methods can be applied to completely  
general multi-marginal optimal
transport problems with pairwise costs.

The relaxed formulation is not exact and provides only a lower
bound to the SCE energy. Hence it is meaningful to compare the error
due to relaxation with the Kohn-Sham SCE model error, i.e., the disparity 
between the Kohn-Sham SCE energy and the exact energy of the solution to the 
quantum many-body problem. Our numerical results for various fermionic lattice 
model problems indicate that the former can be much
smaller than the latter, hence our convex relaxation scheme can be considered to be
effective. On the other hand, as indicated
in, e.g.,~\cite{MaletMirtschinkGiesbertzEtAl2014},
Kohn-Sham SCE is only the zero-th order approximation to the quantum
many-body ground state energy in the limit of large interaction. Hence 
the SCE functional and SCE potential should be considered more properly
as an ``ingredient'' for designing more accurate exchange-correlation
functionals. From such a perspective, just as the exact formulation of SCE is only a model, it may
even be appropriate to consider the relaxed SCE formulation as a
model itself. It can capture certain strong correlation effects and can
be solved efficiently.  

One immediate extension of the current work is to include finite-temperature effects via entropic regularization. In fact, entropic
regularization may be relevant for another reason as well. 
During our numerical studies, we observed that the self-consistent iteration for Kohn-Sham SCE 
(\emph{not} the convex optimization problem solved within each iteration) can 
be difficult to converge. The
convergence behavior may depend sensitively on the filling factor, the
lattice size, and the form of the interaction. Such difficulty can arise for both
the exact SCE formulation solved via linear programming and the relaxed formulations 
solved by SDP. Preliminary results show that entropic regularization can help make 
the loop easier to converge. We are not aware of any reports of such issues in the
literature, and we plan to study such behavior more systematically in
future work. 

\section*{Acknowledgments:} 

This work was partially supported by the Department of Energy under
Grant No. DE-SC0017867, No. DE-AC02-05CH11231, by the Air Force Office
of Scientific Research under award number FA9550-18-1-0095 (L.L.), and
by the National Science Foundation Graduate Research Fellowship Program
under grant DGE-1106400 (M.L.). The work of Y.K. and L.Y. is partially supported by the U.S. Department of Energy, Office of Science, Office of Advanced Scientific Computing Research, Scientific Discovery through Advanced
Computing (SciDAC) program, and the National Science Foundation under award DMS-1818449. We thank Kieron Burke, Gero Friesecke, Paola
Gori-Giorgi, and Michael Seidl for helpful discussions, and the anonymous reviewers for pointing us to important references.

\appendix

\section{Background of KS-DFT}
\label{sec: background}

Our goal is to compute the ground-state energy of a fermionic system with $L$ sites where each site has two states. With some abuse of terminology, we will 
refer to fermions simply as electrons.  Also for simplicity we use a single index 
for all of the states, as opposed to using separate site and spin indices in the case of spinful systems. 
Double indexing for spinful fermionic systems can be recovered simply by rearranging indices, e.g., by associating odd state indices with spin-up components
and even state indices with spin-down components.

In the second-quantized formulation, the state space is called
the Fock space, denoted by $\mathcal{F}$. The occupation number basis
set for the Fock space is 
\[
\{\ket{s_1,\ldots,s_L}\}_{s_i \in
\{0,1\},i=1,\ldots,L},
\]
which is an orthonormal basis set satisfying
\begin{equation}
\braket{s_{i_1},\ldots,s_{i_L} \vert s_{j_1},\ldots,s_{j_L}} =
\delta_{i_1 j_1}\cdots \delta_{i_L j_L}.
\end{equation}
A state $\ket{\psi}\in\mathcal{F}$ will be written as a
linear combination of occupation number basis elements as follows: 
\begin{equation}
  \ket{\psi} = \sum_{s_1,\ldots,s_L\in\{0,1\}}
  \psi(s_1,\ldots,s_L) \ket{s_1,\ldots,s_L}, \quad
  \psi(s_1,\ldots,s_L)\in\CC.
  \label{}
\end{equation}
Hence the state vector $\ket{\psi}$ can be identified with a vector
$\psi\in\CC^{2^{L}}$, and $\mathcal{F}$ is isomorphic to $\CC^{2^{L}}$. We call
$\ket{\psi}$ normalized if the following condition is satisfied:
\begin{equation}
  \braket{\psi\vert\psi} = \sum_{s_1,\ldots,s_L\in\{0,1\}} \vert \psi(s_1,\ldots,s_L)\vert^2 =1.
  \label{eqn:normalization}
\end{equation}
We also refer to $\ket{0}=\ket{0,\ldots,0}$ as the vacuum state.

The fermionic creation and annihilation operators are respectively defined as
\begin{equation}
  \begin{split}
  \hat{a}_p^\dagger \ket{s_1,\ldots,s_L} &= (-1)^{\sum_{q=1}^{p-1}
  s_{q}}(1-s_p)\ket{s_1,\ldots,1-s_{p},\ldots,s_L},\\
  \hat{a}_p \ket{s_1,\ldots,s_L} &= (-1)^{\sum_{q=1}^{p-1} s_{q}}s_p\ket{s_1,\ldots,1-s_{p},\ldots,s_L},\quad p= 1,\ldots,L.
  \end{split}
\end{equation}
The number operator defined as $\hat{n}_{p}:=\hat{a}^{\dagger}_{p} \hat{a}_{p}$ satisfies
\begin{equation}
\hat n_p \ket{s_1,\ldots,s_L} = s_p \ket{s_1,\ldots,s_L},\quad p= 1,\ldots,L.
\end{equation}
The Hamiltonian operator is assumed to take the
following form:
\begin{equation}
\label{eqn:hamiltonian}
\hat{H} = \sum_{p,q=1}^{L} t_{pq} \hat{a}_p^\dagger \hat{a}_q +
\sum_{p=1}^{L} w_{p} \hat{n}_{p} + \sum_{p,q=1}^{L} v_{pq}
\hat{n}_{p} \hat{n}_{q}.
\end{equation}
Here $t\in \CC^{L\times L}$ is a Hermitian matrix, which is often interpreted as the `hopping' term arising from the kinetic energy contribution to the Hamiltonian. $w$ is an
on-site term, which can be viewed as an external potential. 
$v\in \CC^{L\times L}$ is also a Hermitian matrix, which may be viewed as representing the 
electron-electron Coulomb interaction. Note that
$\hat{n}_{p}=\hat{a}_p^\dagger \hat{a}_p=\hat{n}_{p} \hat{n}_{p}$,
hence without loss of generality we can assume the diagonal entries
$t_{pp}=v_{pp}=0$ by absorbing, if necessary, such contributions into the on-site
potential $w$. Following the spirit of Kohn-Sham DFT, one could think of 
$t,v$ as fixed matrices, and of the external potential $w$ as a contribution that may change
depending on the system (in the context of DFT, $w$ represents the electron-nuclei interaction and is therefore `external' to the electrons).  We remark that the restriction of the form of the two-body interaction  $\sum_{p,q=1}^{L} v_{pq} \hat{n}_{p} \hat{n}_{q}$ is crucial for the purpose of this paper. In particular, we do not consider the more general form $\sum_{p,q,r,s=1}^{L} v_{pqrs} \hat{a}^{\dagger}_{p} \hat{a}^{\dagger}_{q} \hat{a}_{s}\hat{a}_r$ as is done in the quantum chemistry literature when a general basis set (such as the Gaussian basis set) is used to discretize a quantum many-body Hamiltonian in the continuous space. In the discussion below, for simplicity we will
omit the index range of our sums as long as the meaning is clear.

The exact ground state energy $E_0$ can be obtained by the following
minimization problem: 
\begin{equation}
  E_{0} = \inf_{\ket{\psi} \in \mathcal{F} \,:\, \braket{\psi\vert\psi}  = 1}
  \bra{\psi} \hat{H}-\mu \hat{N} \ket{\psi}.
  \label{eqn:groundstate}
\end{equation}
Here the minimizer $\ket{\psi}$ is the many-body ground state
wavefunction, and $\hat{N}:=\sum_{p}\hat{n}_{p}$ is the total number
operator. $\mu$, which is called the chemical potential, is a Lagrange
multiplier chosen so that the ground state wavefunction $\vert\psi\rangle$ has a number of electrons equal to a pre-specified integer $N\in \{0,1,\ldots, L\}$, i.e., such that
\begin{equation}
  \braket{\psi\vert \hat{N} \vert \psi}=N.
  \label{eqn:integer_condition}
\end{equation}
It is clear that $\mu \hat{N}$ is an on-site potential, and without loss of
generality we absorb $\mu$ into $w$, and hence write
$\hat{H}-\mu\hat{N}$ as $\hat{H}$ in the discussion below.

The electron density $\rho\in \RR^{L}$ is defined as
\begin{equation}
  \rho_{p} = \braket{\psi\vert \hat{n}_{p} \vert \psi} =
  \sum_{s_{1},\ldots,s_{L}} \abs{\psi(s_{1},\ldots,s_{L})}^2 s_{p}, \quad
  p=1,\ldots, L,
  \label{eqn:density}
\end{equation}
which satisfies $\sum_{p} \rho_{p} = N$.  Note that
\begin{equation}
  \bra{\psi} \sum_{p} w_{p} \hat{n}_{p} \ket{\psi} = \sum_{p}   w_{p} \rho_{p}
 =: W[\rho].
  \label{}
\end{equation}
Then we follow the Levy-Lieb constrained minimization
approach~\cite{Levy1979,Lieb1983} and rewrite the ground state minimization
problem~\eqref{eqn:groundstate} as follows: 
\begin{equation}
  \begin{split}
    E_{0} =& \inf_{\rho\in\mathcal{J}_{N}} \left\{  \sum_{p} \rho_{p} w_{p} + 
    \left(\inf_{\ket{\psi} \mapsto \rho,
    \ket{\psi}\in\mathcal{F}}  \bra{\psi} \sum_{pq} t_{pq}
  \hat a_p^\dagger \hat a_q + \sum_{pq} v_{pq} \hat
  n_{p} \hat n_{q} \ket{\psi} \right) \right\}\\
  = &\inf_{\rho\in\mathcal{J}_{N}} \{ W[\rho] + F_{\text{LL}}[\rho] \}, 
  \end{split}
  \label{eqn:dft}
\end{equation}
where 
\begin{equation}
F_{\text{LL}}[\rho] := \inf_{\ket{\psi} \mapsto \rho,
    \ket{\psi}\in\mathcal{F}}  \bra{\psi} \sum_{pq} t_{pq}
  \hat a_p^\dagger \hat a_q + \sum_{pq} v_{pq} \hat
  n_{p} \hat n_{q} \ket{\psi}.
  \label{eqn:dft2}
\end{equation}
Here the notation $\psi\mapsto\rho$ indicates that the corresponding infimum is taken over states $\ket{\psi}$ that yield the density $\rho$ in the sense of Eq.~\eqref{eqn:density}, and
the domain $\mathcal{J}_{N}$ of $\rho$ is defined by
\begin{equation}
  \mathcal{J}_{N} := \left\{ \rho \in \R^L \ \left\vert\ \rho\ge 0, \,\sum_{p}
  \rho_{p} = N \right. \right\}.
  \label{}
\end{equation}
Note that the external potential $w$ is only coupled with $\rho$ and is singled out in
the constrained minimization. It is easy to see that for
any $\rho\in\mathcal{J}_{N}$, the set
$\{\ket{\psi}\in\mathcal{F}\,:\,\ket{\psi} \mapsto \rho\}$ is non-empty,
as we may simply choose
\[
\ket{\psi}=\sum_{p} \sqrt{\rho_{p}}
\ket{s^{(p)}_1,\ldots,s^{(p)}_L}, \quad s^{(p)}_{q} = \delta_{pq}.
\]
Therefore the constrained minimization problem~\eqref{eqn:dft} is in fact defined over a nonempty set for all $\rho \in \mathcal{J}_N$.

The functional $F_{\text{LL}}[\rho]$, which is called the Levy-Lieb
functional, is a universal functional in the sense that it depends only
on the hopping term $t$ and the interaction term $v$, hence in
particular is independent of the potential $w$. Once the functional
$F_{\text{LL}}[\rho]$ is known, $E_{0}$ can be obtained by minimization
with respect to a single vector $\rho$ using standard optimization
algorithms, or via the self-consistent field (SCF) iteration to be
detailed below.  The construction above is called the `site occupation
functional theory' (SOFT) or `lattice density functional theory' 
in the physics literature~\cite{SchoenhammerGunnarssonNoack1995,LimaOliveiraCapelle2002,CapelleCampo2013,SenjeanNakataniTsuchiizuEtAl2018,Coe2019}.
To our knowledge, SOFT or lattice DFT often imposes an additional sparsity pattern on the $v$ matrix for the electron-electron interaction, so that the Hamiltonian becomes a Hubbard-type model.  

\subsection{Strictly correlated electron limit}
Using the fact that the infimum of a sum is greater than the sum of infimums, we can lower-bound
the ground state energy in the following way: 
\begin{equation}
  \label{eqn:lower_bound}
  F_{\text{LL}}[\rho] \geq\inf_{\ket{\psi} \mapsto \rho}
    \bra{\psi} \sum_{pq} t_{pq} \hat{a}_p^\dagger
    \hat{a}_q \ket{\psi} + \inf_{\ket{\psi} \mapsto \rho}
    \bra{\psi}\sum_{pq} v_{pq}\hat n_{p} \hat n_{q}
    \ket{\psi} =: T[\rho] + E_{\text{sce}}[\rho],
\end{equation}
where the functionals $T[\rho]$ and $E_{\text{sce}}[\rho]$ are defined via the last equality in the manner suggested by the notation. The first of these quantities is called the kinetic energy, and the second the
strictly correlated electron (SCE) energy. The SCE approximation is obtained
by treating $T[\rho]+E_{\text{sce}}[\rho]$ as an approximation for the
Levy-Lieb functional. Though in general it is only a lower-bound for the
Levy-Lieb functional, this bound is expected to become tight in the
limit of infinitely strong interaction. We do not prove this fact in
this paper (though we demonstrate it numerically below), but we
nonetheless refer to this approximation as the SCE limit by analogy to the
literature on SCE in first quantization~\cite{SeidlPerdewLevy1999,Seidl2007}.

Due to the inequality in
Eq.~\eqref{eqn:lower_bound}, we have in general the following lower bound for the total energy, which we shall call the Kohn-Sham SCE energy: 
\begin{equation}
  E_{0} \ge E_{\textrm{KS-SCE}} := \inf_{\rho\in\mathcal{J}_{N}} \left\{ W[\rho] + T[\rho] + E_{\text{sce}}[\rho]
   \right\}.
  \label{eqn:E_Kohn-Sham SCE}
\end{equation}
The advantage of the preceding manipulations is that now each term in this infimum can be computed. Specifically, $W[\rho]$ is trivial to compute, $T[\rho]$ is defined in terms of a non-interacting many-body problem (i.e., a problem with Hamiltonian only quadratic in the creation and annihilation operators), for which an exact solution can be obtained via the diagonalization of $t$ \cite{NegeleOrland1988}. Finally, as we shall see below the SCE term (and its gradient) can be computed in terms of a MMOT problem (and its dual). Thus in principle, it would be possible to take gradient descent approach for computing the infimum in the definition \eqref{eqn:E_Kohn-Sham SCE} of $E_{\textrm{KS-SCE}}$.

\subsubsection{The Kohn-Sham SCE equations} \label{sec:kssceEq}

In practice, to compute the Kohn-Sham SCE energy we will instead adopt
the self-consistent field (SCF) iteration as is common practice in
Kohn-Sham DFT.  It can be readily checked that $E_{\text{sce}}[\rho]$ is convex with respect to $\rho$. By the convexity of $W[\rho]$, $T[\rho]$, and $E_{\text{sce}}[\rho]$, the expression in Eq.~\eqref{eqn:E_Kohn-Sham SCE} admits a minimizer, which is unique unless the functional fails to be strictly convex. 
We assume that the solution is unique and $E_{\text{sce}}[\rho]$ is differentiable for simplicity, and we derive nonlinear fixed-point equations satisfied by the minimizer as follows. 

For suitable $\rho$, define the SCE potential via 
\begin{equation}
  v_{\text{sce}}[\rho] = \nabla_\rho E_{\text{sce}} [\rho].
  \label{eqn:SCE_potential}
\end{equation}
Now assume that the (unique) infimum in Eq.~\eqref{eqn:E_Kohn-Sham SCE} is obtained at $\rho^\star$, which is then in particular a critical point of the expression
\begin{equation}
W[\rho] + T[\rho] + E_{\text{sce}}[\rho].
\label{eqn:infExp}
\end{equation}
But then $\rho^\star$ is also a critical point of the expression obtained by replacing $E_{\text{sce}}[\rho]$ with its expansion up to first order about $\rho^\star$, which is (modulo a constant term that does not affect criticality) 
\begin{equation}
G[\rho] := W[\rho] + T[\rho] + v_{\text{sce}}[\rho^\star]\cdot \rho = T[\rho] + (w+v_{\text{sce}}[\rho^\star])\cdot\rho.
\label{eqn:Gdef}
\end{equation}
Hence $\cdot$ means the inner product, and we are motivated to try to
minimize $G[\rho]$ over $\rho \in \mathcal{J}_N$.  But we can write
\[
G[\rho] = \inf_{\ket{\psi} \mapsto \rho}
    \bra{\psi} \sum_{pq} h_{pq}[\rho^\star] \hat{a}_p^\dagger
    \hat{a}_q \ket{\psi},
\]
where 
\[
h[\rho] := t + \mathrm{diag}(w+v_{\text{sce}}[\rho]).
\]
Here $\mathrm{diag}(\cdot)$ is a diagonal matrix.  Then
\[
\inf_{\rho \in \mathcal{J}_N} G[\rho] = 
\inf_{\ket{\psi}\in \mathcal{F} \,:\,\langle \psi \vert \psi \rangle =1,\,\langle \psi \vert \hat{N} \vert \psi \rangle = N }
    \bra{\psi} \sum_{pq} h_{pq}[\rho^\star] \hat{a}_p^\dagger
    \hat{a}_q \ket{\psi}.
\]
The latter infimum is a ground-state problem for a non-interacting Hamiltonian and is obtained \cite{NegeleOrland1988} at a so-called Slater determinant of the form
\begin{equation}
  \ket{\psi} = \hat{c}^{\dagger}_{1} \cdots \hat{c}^{\dagger}_{N}
  \ket{0}.
  \label{eqn:slater}
\end{equation}
Here the $c_k^\dagger$ are `canonically transformed' creation operators defined by 
\begin{equation}
  \hat{c}^{\dagger}_{k} = \sum_{p} \hat{a}^{\dagger}_{p} \varphi_{pk},
  \label{}
\end{equation}
where 
$\Phi=[\varphi_1 \cdots \varphi_N ] = [\varphi_{pk}]\in\CC^{L\times N}$
is a matrix whose columns are the $N$ lowest eigenvectors of
$h[\rho^\star]$. We assume the eigenvectors form an orthonormal set,
i.e. $\Phi^* \Phi = I_N$.

Moreover, one may directly compute that the electron density of $\vert\psi\rangle$ as defined in Eq.~\eqref{eqn:slater} is given by
\begin{equation}
  \rho_{p} = \braket{\psi\vert \hat{n}_{p} \vert \psi} = \sum_{k=1}^{N} \abs{\varphi_{pk}}^2,
  \label{eqn:density_slater}
\end{equation}
i.e., $\rho = \mathrm{diag}(\Phi \Phi^*)$.
Hence the optimizer $\rho^\star$ of Eq.~\eqref{eqn:E_Kohn-Sham SCE} solves the Kohn-Sham SCE equations: 
\begin{equation}
  \begin{split}
(t + \mathrm{diag}(w+v_{\text{sce}}[\rho])) \varphi_i &= \varepsilon_i
\varphi_i, \quad  i=1,\ldots, N.\vspace{1 mm} \\
\rho &= \mathrm{diag}(\Phi \Phi^*).
  \end{split}
\label{eqn:KSSCE_eqn 2}
\end{equation}
Here $(\ve_i,\varphi_i)$ are understood to be the $N$ lowest (orthonormal) eigenpairs of the matrix in the first line of Eq.~\eqref{eqn:KSSCE_eqn 2}. Let $\rho_\star$ be a solution to \eqref{eqn:KSSCE_eqn 2}, the total energy can be recovered by the
relation 
\begin{equation}
  E_{\text{KS-SCE}} = \sum_{k=1}^{N} \varepsilon_{k} - \nabla_\rho E_{\text{sce}}[\rho^\star]^T \rho^\star + E_{\text{sce}}[\rho^\star].
\end{equation}
as can be observed by adding back to $G[\rho^\star]$ the constant term
discarded between equations~\eqref{eqn:infExp} and \eqref{eqn:Gdef}.

\subsubsection{The SCE energy and potential}
The problem is then reduced to the computation of $E_{\text{sce}}[\rho]$
and its gradient $v_{\text{sce}}[\rho]$.
To this end, let us rewrite
\begin{equation}\label{eqn:esce_rewrite}
  \begin{split}
    E_{\text{sce}}[\rho] =& \inf_{\ket{\psi} \mapsto \rho}  
  \bra{\psi}\sum_{pq} v_{pq} \hat n_{p} \hat n_{q}
  \ket{\psi} \\
  =&\inf_{\ket{\psi} \mapsto \rho}  \sum_{s_1,\ldots,s_L} \sum_{pq} v_{pq} s_{p} s_{q} \vert \psi(s_1,\ldots,s_L)\vert^2\\
  =&\inf_{\mu \in \Pi(\rho)} \sum_{s_1,\ldots,s_L}
  \sum_{pq} v_{pq} s_{p}  s_{q} \mu(s_1,\ldots,s_L),
  \end{split}
\end{equation}
The last line of Eq.~\eqref{eqn:esce_rewrite} is obtained by considering
$\vert \psi(s_1,\ldots,s_L)\vert^2$ as a
classical probability density $\mu(s_1,\ldots,s_L) \in \Pi(\rho)$. (The marginal condition derives from the condition $\vert\Psi\rangle \mapsto \rho$.)

\bibliographystyle{siam}
\bibliography{sce}

\begin{thebibliography}{10}

\bibitem{Anderson1965}
{\sc D.~G. Anderson}, {\em {Iterative procedures for nonlinear integral
  equations}}, J. Assoc. Comput. Mach., 12 (1965), pp.~547--560.

\bibitem{Becke1988}
{\sc A.~D. Becke}, {\em {Density-functional exchange-energy approximation with
  correct asymptotic behavior}}, Phys. Rev. A, 38 (1988), pp.~3098--3100.

\bibitem{BenamouCarlierNenna2016}
{\sc J.-D. Benamou, G.~Carlier, and L.~Nenna}, {\em A numerical method to solve
  multi-marginal optimal transport problems with {C}oulomb cost}, in Splitting
  Methods in Communication, Imaging, Science, and Engineering, Springer, 2016,
  pp.~577--601.

\bibitem{BoydVandenberghe2004}
{\sc Stephen Boyd and Lieven Vandenberghe}, {\em Convex optimization},
  Cambridge Univ. Pr., 2004.

\bibitem{ButtazzoDePascaleGori-Giorgi2012}
{\sc Giuseppe Buttazzo, Luigi {De Pascale}, and Paola Gori-Giorgi}, {\em
  {Optimal-transport formulation of electronic density-functional theory}},
  Phys. Rev. A, 85 (2012), p.~062502.

\bibitem{CapelleCampo2013}
{\sc Klaus Capelle and Vivaldo~L. Campo}, {\em {Density functionals and model
  Hamiltonians: Pillars of many-particle physics}}, Phys. Rep., 528 (2013),
  pp.~91--159.

\bibitem{ChenFrieseckeMendl2014}
{\sc Huajie Chen, Gero Friesecke, and Christian~B Mendl}, {\em Numerical
  methods for a kohn-sham density functional model based on optimal transport},
  J. Chem. Theory Comput., 10 (2014), pp.~4360--4368.

\bibitem{Coe2019}
{\sc J.~P. Coe}, {\em Lattice density-functional theory for quantum chemistry},
  Phys. Rev. B, 99 (2019), p.~165118.

\bibitem{Coleman1963}
{\sc A~John Coleman}, {\em Structure of fermion density matrices}, Rev. Mod.
  Phys., 35 (1963), p.~668.

\bibitem{colombo2019continuity}
{\sc Maria Colombo, Simone Di~Marino, and Federico Stra}, {\em Continuity of
  multimarginal optimal transport with repulsive cost}, SIAM Journal on
  Mathematical Analysis, 51 (2019), pp.~2903--2926.

\bibitem{CotarFrieseckeKlueppelberg2013}
{\sc Codina Cotar, Gero Friesecke, and Claudia Kl{\"u}ppelberg}, {\em Density
  functional theory and optimal transportation with coulomb cost}, Commun. Pure
  Appl. Math., 66 (2013), pp.~548--599.

\bibitem{cotar2018smoothing}
\leavevmode\vrule height 2pt depth -1.6pt width 23pt, {\em Smoothing of
  transport plans with fixed marginals and rigorous semiclassical limit of the
  {H}ohenberg--{K}ohn functional}, Archive for Rational Mechanics and Analysis,
  228 (2018), pp.~891--922.

\bibitem{de2015optimal}
{\sc Luigi De~Pascale}, {\em Optimal transport with coulomb cost. approximation
  and duality}, ESAIM: Mathematical Modelling and Numerical Analysis, 49
  (2015), pp.~1643--1657.

\bibitem{di2017optimal}
{\sc Simone Di~Marino, Augusto Gerolin, and Luca Nenna}, {\em Optimal
  transportation theory with repulsive costs}, Topological optimization and
  optimal transport, 17 (2017), pp.~204--256.

\bibitem{friesecke2013n}
{\sc Gero Friesecke, Christian~B Mendl, Brendan Pass, Codina Cotar, and Claudia
  Kl{\"u}ppelberg}, {\em N-density representability and the optimal transport
  limit of the {H}ohenberg-{K}ohn functional}, The Journal of chemical physics,
  139 (2013), p.~164109.

\bibitem{friesecke2018breaking}
{\sc Gero Friesecke and Daniela V\"{o}gler}, {\em Breaking the curse of
  dimension in multi-marginal kantorovich optimal transport on finite state
  spaces}, SIAM Journal on Mathematical Analysis, 50 (2018), pp.~3996--4019.

\bibitem{gerolin2019duality}
{\sc Augusto Gerolin, Anna Kausamo, and Tapio Rajala}, {\em Duality theory for
  multi-marginal optimal transport with repulsive costs in metric spaces},
  ESAIM: Control, Optimisation and Calculus of Variations, 25 (2019), p.~62.

\bibitem{grant2008cvx}
{\sc M.~Grant and S.~Boyd}, {\em {CVX}: Matlab software for disciplined convex
  programming}, 2013.

\bibitem{GrossiKooiGiesbertzEtAl2017}
{\sc Juri Grossi, Derk~P Kooi, Klaas~JH Giesbertz, Michael Seidl, Aron~J Cohen,
  Paula Mori-S{\'a}nchez, and Paola Gori-Giorgi}, {\em Fermionic statistics in
  the strongly correlated limit of density functional theory}, J. Chem. Theory
  Comput., 13 (2017), pp.~6089--6100.

\bibitem{HohenbergKohn1964}
{\sc P.~Hohenberg and W.~Kohn}, {\em {Inhomogeneous electron gas}}, Phys. Rev.,
  136 (1964), pp.~B864--B871.

\bibitem{KhooYingOT}
{\sc Y.~Khoo and L.~Ying}, {\em Convex relaxation approaches for strictly
  correlated density functional theory}, arXiv:1808.04496,  (2018).

\bibitem{KohnSham1965}
{\sc W.~Kohn and L.~Sham}, {\em {Self-consistent equations including exchange
  and correlation effects}}, Phys. Rev., 140 (1965), pp.~A1133--A1138.

\bibitem{Komiya1988}
{\sc H.~Komiya}, {\em Elementary proof for {Sion's} minimax theorem}, Kodai
  Math. J., 11 (1988), pp.~5--7.

\bibitem{LeeYangParr1988}
{\sc C.~Lee, W.~Yang, and R.~G. Parr}, {\em {Development of the Colle-Salvetti
  correlation-energy formula into a functional of the electron density}}, Phys.
  Rev. B, 37 (1988), pp.~785--789.

\bibitem{Levy1979}
{\sc M.~Levy}, {\em Universal variational functionals of electron densities,
  first-order density matrices, and natural spin-orbitals and solution of the
  v-representability problem}, Proc. Natl. Acad. Sci., 76 (1979),
  pp.~6062--6065.

\bibitem{lewin2018semi}
{\sc Mathieu Lewin}, {\em Semi-classical limit of the levy--lieb functional in
  density functional theory}, Comptes Rendus Mathematique, 356 (2018),
  pp.~449--455.

\bibitem{Lieb1983}
{\sc E.~H. Lieb}, {\em Density functional for {Coulomb} systems}, Int J.
  Quantum Chem., 24 (1983), p.~243.

\bibitem{LimaOliveiraCapelle2002}
{\sc N.~A. Lima, L.~N. Oliveira, and K.~Capelle}, {\em {Density-functional
  study of the Mott gap in the Hubbard model}}, Europhys. Lett., 60 (2002),
  pp.~601--607.

\bibitem{LinYang2013}
{\sc L.~Lin and C.~Yang}, {\em {Elliptic preconditioner for accelerating self
  consistent field iteration in Kohn-Sham density functional theory}}, SIAM J.
  Sci. Comp., 35 (2013), pp.~S277--S298.

\bibitem{MaletGori-Giorgi2012}
{\sc F.~Malet and P.~Gori-Giorgi}, {\em {Strong Correlation in Kohn-Sham
  Density Functional Theory}}, Phys. Rev. Lett., 109 (2012), p.~246402.

\bibitem{MaletMirtschinkGiesbertzEtAl2014}
{\sc Francesc Malet, Andr{\'e} Mirtschink, Klaas~JH Giesbertz, Lucas~O Wagner,
  and Paola Gori-Giorgi}, {\em Exchange--correlation functionals from the
  strong interaction limit of {DFT}: applications to model chemical systems},
  Phys. Chem. Chem. Phys., 16 (2014), pp.~14551--14558.

\bibitem{Mazziotti2004}
{\sc David Mazziotti}, {\em Realization of quantum chemistry without wave
  functions through first-order semidefinite programming}, Phys. Rev. Lett., 93
  (2004), p.~213001.

\bibitem{Mazziotti2012}
\leavevmode\vrule height 2pt depth -1.6pt width 23pt, {\em Structure of
  fermionic density matrices: Complete {N}-representability conditions}, Phys.
  Rev. Lett., 108 (2012), p.~263002.

\bibitem{Mazziotti1998}
{\sc David~A Mazziotti}, {\em Contracted {S}chr{\"o}dinger equation:
  Determining quantum energies and two-particle density matrices without wave
  functions}, Phys. Rev. A, 57 (1998), p.~4219.

\bibitem{mcclean2017openfermion}
{\sc J.~R. McClean et~al.}, {\em {OpenFermion}: the electronic structure
  package for quantum computers}, arXiv:1710.07629,  (2017).

\bibitem{MendlLin2013}
{\sc C.~Mendl and L.~Lin}, {\em Kantorovich dual solution for strictly
  correlated electrons in atoms and molecules}, Phys. Rev. B, 87 (2013),
  p.~125106.

\bibitem{MendlMaletGori-Giorgi2014}
{\sc Christian~B Mendl, Francesc Malet, and Paola Gori-Giorgi}, {\em Wigner
  localization in quantum dots from kohn-sham density functional theory without
  symmetry breaking}, Phys. Rev. B, 89 (2014), p.~125106.

\bibitem{nakata2001variational}
{\sc Maho Nakata, Hiroshi Nakatsuji, Masahiro Ehara, Mitsuhiro Fukuda, Kazuhide
  Nakata, and Katsuki Fujisawa}, {\em Variational calculations of fermion
  second-order reduced density matrices by semidefinite programming algorithm},
  The Journal of Chemical Physics, 114 (2001), pp.~8282--8292.

\bibitem{NegeleOrland1988}
{\sc J.~W. Negele and H.~Orland}, {\em Quantum many-particle systems},
  Westview, 1988.

\bibitem{PerdewBurkeErnzerhof1996}
{\sc J.~P. Perdew, K.~Burke, and M.~Ernzerhof}, {\em {Generalized gradient
  approximation made simple}}, Phys. Rev. Lett., 77 (1996), pp.~3865--3868.

\bibitem{PerdewZunger1981}
{\sc J.~P. Perdew and A.~Zunger}, {\em {Self-interaction correction to
  density-functional approximations for many-electron systems}}, Phys. Rev. B,
  23 (1981), pp.~5048--5079.

\bibitem{Pulay1980}
{\sc P.~Pulay}, {\em {Convergence acceleration of iterative sequences: The case
  of SCF iteration}}, Chem. Phys. Lett., 73 (1980), pp.~393--398.

\bibitem{raghu2010superconductivity}
{\sc S~Raghu, SA~Kivelson, and DJ~Scalapino}, {\em Superconductivity in the
  repulsive {H}ubbard model: An asymptotically exact weak-coupling solution},
  Physical Review B, 81 (2010), p.~224505.

\bibitem{rock}
{\sc R.~T. Rockafellar}, {\em Convex analysis}, Princeton University Press,
  1970.

\bibitem{SchoenhammerGunnarssonNoack1995}
{\sc K.~Sch{\"o}nhammer, O.~Gunnarsson, and R.~M. Noack}, {\em
  Density-functional theory on a lattice: Comparison with exact numerical
  results for a model with strongly correlated electrons}, Phys. Rev. B, 52
  (1995), p.~2504.

\bibitem{Seidl2007}
{\sc M.~Seidl, P.~Gori-Giorgi, and A.~Savin}, {\em Strictly correlated
  electrons in density-functional theory: A general formulation with
  applications to spherical densities}, Phys. Rev. A, 75 (2007), p.~042511.

\bibitem{SeidlPerdewLevy1999}
{\sc Michael Seidl, John~P Perdew, and Mel Levy}, {\em Strictly correlated
  electrons in density-functional theory}, Phys. Rev. A, 59 (1999), p.~51.

\bibitem{SenjeanNakataniTsuchiizuEtAl2018}
{\sc Bruno Senjean, Naoki Nakatani, Masahisa Tsuchiizu, and Emmanuel Fromager},
  {\em {Multiple impurities and combined local density approximations in
  site-occupation embedding theory}}, Theor. Chem. Acc., 137 (2018), pp.~1--21.

\bibitem{villani2008optimal}
{\sc C.~Villani}, {\em Optimal transport: old and new}, Springer, 2009.

\bibitem{vogler2019kantorovich}
{\sc Daniela V{\"o}gler}, {\em Kantorovich vs. {M}onge: A numerical
  classification of extremal multi-marginal mass transports on finite state
  spaces}, arXiv preprint arXiv:1901.04568,  (2019).

\end{thebibliography}

\end{document}